\documentclass[12pt]{article}
\usepackage{amsfonts,amsmath,amsthm,amssymb}
\usepackage{xcolor,graphicx}
\usepackage{comment}
\usepackage{microtype}
\usepackage{hyperref}

\usepackage{currfile}
\usepackage{navigator}
\IfFileExists{\currfilename}{\embeddedfile{sourcefile}{\currfilename}}{}
 \hypersetup{pdftitle={Completely regular codes in graphs covered by a Hamming graph},
              pdfsubject={completely regular codes, Hamming graphs, Cayley graphs, strongly regular graphs, graph covering, orthogonal arrays},
              pdfauthor={S. Goryainov, D. Krotov}}

\title{Completely regular codes in graphs\\ covered by a Hamming graph%
\thanks{This is a post-peer-review, pre-copyedit version of the article published in Designs, Codes and Cryptography
94(10), 2026, 213(1--25), \url{https://doi.org/10.1007/s10623-026-01955-2}.

The work of S. Goryainov 
is supported by the Natural Science Foundation of Hebei Province (A2023205045) and the 111 Center (Grant No.D26018). The work of  D. Krotov is supported within the framework of the state contract of the Sobolev Institute of
Mathematics (FWNF-2026-0011).}}
\author{Sergey Goryainov%
\thanks{School of Mathematical Sciences, Hebei International Joint Research Center for Mathematics and Interdisciplinary Science, Hebei Key Laboratory of Computational Mathematics and Applications, Hebei Workstation for Foreign Academicians, Hebei Normal University, Shijiazhuang 050024, P.R. China}\ \ and Denis Krotov%
\thanks{Sobolev Institute of Mathematics, 
Novosibirsk 630090, Russia}}
\date{}

\newcommand{\FF}{\mathbb{F}}
\newcommand{\ZZ}{\mathbb{Z}}
\newcommand{\Aut}{\mathrm{Aut}}
\newcommand{\Cay}{\mathrm{Cay}}
\newcommand{\vx}{}

\newtheorem{proposition}{Proposition}
\newtheorem{lemma}{Lemma}
\newtheorem{theorem}{Theorem}
\newtheorem{corollary}{Corollary}
\theoremstyle{remark}
\newtheorem{remark}{Remark}
\newtheorem{example}{Example}
\theoremstyle{definition}
\newtheorem{definition}{Definition}
\newtheorem{conjecture}{Conjecture}

\begin{document}
\maketitle
\begin{abstract}
In Cayley graphs on the additive group of a small vector space over GF$(q)$, $q=2,3$, we look for completely regular (CR) codes whose parameters are new in Hamming graphs over the same field. The existence of a CR code in such a Cayley graph $G$ implies the existence of a CR code with the same parameters in the corresponding Hamming graph that covers $G$. In such a way, we find several completely regular codes with new parameters in Hamming graphs over GF$(2)$ and GF$(3)$. The most interesting findings are two new CR-$1$ (with covering radius~$1$) codes that are independent sets (such CR are equivalent to optimal orthogonal arrays attaining the Bierbrauer--Friedman bound) and one new CR-$2$ code. By recursive constructions, every new CR code induces an infinite sequence of CR codes (in particular, optimal orthogonal arrays if the original code is CR-$1$ and independent). Some of the found ternary CR codes disprove a conjectured generalization of the correlation-immunity bound on the strength of binary orthogonal arrays to the non-binary case. In between, we classify feasible parameters of CR codes in several strongly regular graphs.


\end{abstract}

\section{Introduction}

The main result of this paper
is the discovery of nonlinear completely regular (CR) codes with new parameters
in Hamming graphs $H(n,q)$, $q=2,3$.
The summary with new parameters can be found in 
Section~\ref{s:sum}. 
They update the tables of small parameters of CR codes
in~\cite{KooKroMar:Ch7tables}. 
Most of the new parameters update the smallest~$n$ for which
CR codes with given intersection array are known to exist in $H(n,3)$. However, the intersection arrays of a discovered $\{21,4;2,21\}$-CR code in $H(11,3)$, a constructed $\{66,15;1,54\}$-CR code in $H(66,2)$, and families of CR codes of covering radius~$2$ obtained with a new construction from $\{22;5\}$-CR
and $\{50;4\}$-CR codes, are completely new.
We note that there are very few known parameters of CR codes of covering radius more than $1$ that cannot be constructed as linear codes, see the survey of constructions in~\cite{KooKroMar:Ch7tables} and one of the results in~\cite{Krotov:2024:CRexist}, where a $\{13, 6, 1; 1, 6, 9\}$-CR code in $H(13,2)$ was found.

CR codes of covering radius $1$  that are independent sets in Hamming graphs correspond to optimal orthogonal arrays attaining the Bierbrauer--Friedman bound, see e.g. \cite[Sect.\,1.4.3.1]{KroPot:Ch1}. 
We have found two new such codes. 
They are the first such nonbinary codes whose cardinality
is not a power of the prime divisor of~$q$ (and so, codes with the same parameters cannot be constructed 
as linear or additive codes).
Because of recursive constructions described in Section~\ref{s:constr}, now we have an infinite family
of such codes and corresponding orthogonal arrays.

In the next section, we give the necessary definitions and basic facts.
In Section~\ref{s:constr} we show several constructions of CR codes;
one of them (Theorem~\ref{th:1to2}) is new.
In Section~\ref{s:CRinCay} with many subsections 
we describe new parameters of completely regular codes in Hamming graphs and all parameters of completely regular codes in some strongly regular graphs. In Section~\ref{s:sum} we summarize the results in a table and discuss some conclusions.

\section{Preliminaries}

In this paper, all graphs are undirected and simple, without loops or multiple edges. 
However, some introduced concepts and 
claims are valid for multigraphs,
see the details e.g. in \cite{KroPot:Ch1}.

As usual, $\FF_q$ denotes the finite field
of prime power order $q$ 
and $\FF_q^n$ denotes the vector space
of $n$-tuples ($n$-words) over~$\FF_q$.
We mostly consider
the cases $q=2$ and $q=3$; 
in particular, Cayley graphs,
defined below,
are considered over the additive group of vector spaces over $\FF_q$, $q=2,3$.

\newcommand{\Id}{\mathrm{Id}}
\begin{definition}[\textbf{Cayley graph}]
\label{1def:Cayley}
The \emph{Cayley graph} $\Cay(G,S)$ on a group $G$ (written in the multiplicative notation)
with \emph{connecting} set
$S$, $S\subset G\backslash\{\Id\}$, $S=S^{-1}$,
is the graph whose vertices
are the elements of~$G$ and whose edge set~$E$ is $\{\{x,ax\}:\ x\in G, a\in S\}$
(equivalently, $\{x,y\}\in E$ if and only if $y\cdot x^{-1} \in S$).
\end{definition}

\begin{definition}[\textbf{\boldmath 
Hamming graph
}]
\label{1def:Hamming}
The \emph{complete graph} $K_q$ is the simple graph of order~$q$ where every two different vertices are adjacent.
The \emph{Hamming graph} $H(n,q)$ is the Cartesian product of $n$ copies of~$K_q$.
According to the definition of the Cartesian product, the vertex set of $H(n,q)$
consists of all $n$-words in some alphabet of order~$q$, which,
if $q$ is a prime power, is often associated with the elements of the Galois field
$\FF_q=\mathrm{GF}(q)$ of order~$q$.
In this case, the vertex set of $H(n,q)$ forms an $n$-dimensional vector space~$\FF_q^n$ over~$\FF_q$,
and $H(n,q)$ is a Cayley graph on its additive group~$\FF_q^{n+}$. The natural minimum-path distance in a Hamming graph is called the \emph{Hamming distance}.
\end{definition}

\begin{definition}[\textbf{\boldmath clique, coclique = independent set}]
\label{1def:clique}
In a graph, a set of vertices that are mutually adjacent (mutually non-adjacent) to each other
is called a \emph{clique} (\emph{coclique}, or \emph{independent set}) or, more specifically,
\emph{$k$-clique} (\emph{$k$-coclique}) where $k$ is the cardinality of this set.
\end{definition}

\begin{definition}[\textbf{\boldmath perfect coloring, equitable partition, quotient matrix}]
\label{1def:equitable}
A \emph{$k$-coloring} (or simply \emph{coloring})
of a graph $\Gamma$ is a surjective map from the set of vertices of $\Gamma$ to a set~$\Sigma$ of size~$k$, whose elements are referred to as \emph{colors}.
A $k$-coloring of the vertices of a graph~$\Gamma$ is called \emph{perfect}
if there is a $k$-by-$k$
matrix $(S_{i,j})_{i,j\in \Sigma}$
(the \emph{quotient} matrix)
whose rows and columns are indexed by colors,
such that every vertex of color~$i$
has exactly~$S_{i,j}$ neighbors
of color~$j$. An ordered partition $(C_0, \ldots, C_{k-1})$
of the vertex set of~$\Gamma$ into nonempty sets
$C_0$, \ldots, $C_{k-1}$ is called \emph{equitable},
or an \emph{equitable $k$-partition}, if the vertex coloring~$f$
defined as $f(x)=i$ if $x\in C_i$ is perfect.
We will also use
the shorter forms
``$S$-perfect coloring'' for ``perfect coloring with quotient matrix~$S$''
and 
``equitable $S$-partition'' for ``equitable partition with quotient matrix~$S$''.
\end{definition}

It is well known that the eigenvalues
of the quotient matrix of an equitable partition are always eigenvalues of the graph (i.e., of its adjacency matrix).
If the graph is regular, the largest eigenvalue of the quotient matrix is always the degree of the graph. By the eigenvalue of an equitable $2$-partition we will mean the non-largest eigenvalue of its quotient matrix.

\begin{definition}[\textbf{\boldmath covering radius, completely regular (CR) code}]\label{1d:crc}
\label{1def:crc}
Given a nonempty set~$C$ of vertices of a connected graph~$\Gamma=(V,E)$,
the \emph{distance partition} with respect to~$C$
is the partition $(C^{(0)},C^{(1)},\ldots ,C^{(\rho )})$ of~$V$
with respect to the distance from~$C$: $C^{(i)}=\{\vx x\in V:
\ d(\vx x,C)=i\}$. The number~$\rho=\rho(C)$ of
non-empty cells different from~$C$
in the distance partition
is called the
\emph{covering radius} of~$C$.
A nonempty set~$C$ (code) of vertices of a graph
is called
\emph{completely regular} (\emph{CR}) (completely regular set or completely regular code) 
if the distance partition  
with respect to~$C$ is equitable,
or, equivalently,
if the coloring $f(x)=d(x,C)$ of the vertex set is perfect.
The quotient matrix
$(S_{i,j})_{i,j=0}^\rho$
corresponding
to a completely regular code
is always tridiagonal, and for regular graphs it is often written
in the form of the \emph{intersection array}
$$\{S_{0,1},S_{1,2},\ldots,S_{\rho-1,\rho}; S_{1,0},S_{2,1},\ldots,S_{\rho,\rho-1}\}$$
(the diagonal elements of the quotient matrix are uniquely determined from
the intersection array and the degree of the graph).

We will use the abbreviation ``$A$-CR'' for ``completely regular with intersection array~$A$''
or ``completely regular with quotient matrix~$A$'',
depending on the context, and 
the abbreviation ``CR-$\rho$'' for 
``completely regular with covering radius~$\rho$''.
\end{definition}

A very important special case
of perfect colorings
is graph coverings.

\begin{definition}[\textbf{\boldmath graph covering}]\label{1d:gr-cov}
\label{1def:covering}
Assume that $f$ is a perfect coloring of a graph~$\Gamma$
with symmetric zero-diagonal quotient $\{0,1\}$-matrix~$S$ and $\overline P$ is the equitable partition corresponding to~$f$.
Then $S$ can be considered as the adjacency matrix of
some graph, denoted $\Gamma/\overline P$, on the set of colors
and we say that the coloring~$f$,
as well as the corresponding partition~$\overline P$,
is a \emph{covering} of the \emph{target} graph~$\Gamma/\overline P$
by a \emph{cover} graph~$\Gamma$.
In other words, a covering is a locally bijective
(bijective on the set of edges incident to a vertex)
surjective homomorphism 
from one (cover) graph to another (target) graph.
\end{definition}

\begin{lemma}[a special case of {\cite[Lemma 1.1]{KroPot:Ch1}}]\label{1l:cay-cov}
    A surjective homomorphism~$\varphi$ 
    from a group~$G$ onto a group~$\varphi(G)$
    is a covering of $\Cay(\varphi(G),\varphi(S))$
    by $\Cay(G,S)$, for any set~$S$ of elements of~$G$ such that $S^{-1}=S$, $\mathrm{Id}\not\in\varphi(S)$, and $|\varphi(S)|=|S|$.
\end{lemma}

\begin{corollary}
    Any connected Cayley graph $\Cay(\FF_q^{k+},S)$ on the additive group of $\FF_q^{k}$, $q\in\{2,3\}$,
    is covered by $H(|S|/(q-1),q)$. Moreover, the corresponding covering is a homomorphism
    of $\FF_q^{n}$, $n=|S|/(q-1)$, onto $\FF_q^{k}$.
\end{corollary}

It is also known that 
\begin{lemma}[see e.g. {\cite[Proposition 1.36(i)]{KroPot:Ch1}}]\label{1l:conc}
The existence
    of an equitable $S$-partition of a graph $G$ implies
the existence of an equitable $S$-partition
of any graph~$G$ that covers~$G$.
\end{lemma}



\subsection{Linear codes and Cayley graphs}\label{s:LinCay}

\begin{definition}[\textbf{linear codes, generator and check matrices, equivalent codes}]
Any nontrivial subspace~$C$ of~$\FF_q^n$ is called a \emph{linear $q$-ary code} of length~$n$, 
dimension $k=\dim(C)$, 
and (minimum) distance~$d$ 
if $d$ is 
 the minimum Hamming distance between two distinct elements (codewords) of the code.
 A linear code $C\subseteq\FF_q^n$ with distance at least~$d$ is called an $[n,k,d]_q$ code. 
A matrix whose rows generate a linear code~$C$ is called a \emph{generator} matrix of~$C$.
A matrix whose null-space is $C$ is called 
a \emph{check} matrix of~$C$. 
Two linear codes are \emph{equivalent} if one of them can be sent to the other by a permutation of coordinates, multiplication
of each position by a nonzero element of~$\FF_q$, and applying an automorphism of the field~$\FF_q$ simultaneously to all positions 
(for prime~$q$, considered in the paper, there are no nontrivial automorphisms of~$\FF_q$).
\end{definition}

\begin{definition}[\textbf{syndrome graph}]
Given a check $m\times n$ matrix $H$
of a $[n,n-m,3]_q$ code $C$
(distance at least~$3$ is equivalent to saying that 
no two distinct columns of~$H$ are collinear,
in particular, no column is all-zero),
its \emph{syndrome graph}
is the Cayley graph on $\FF_q^m$ with the connecting
set $\{\alpha h:\ h\in\mathrm{columns}(H),\ \alpha\in\FF_q^*\}$.
\end{definition}

Formally, the syndrome graph of a code depends on the choice of the check matrix; however, all syndrome graphs of the same code are isomorphic to the so-called \emph{coset graph} of the code~$C$, where the vertices are cosets of~$C$ and two cosets are adjacent if they have adjacent vertices in~$H(n,q)$.

It is not difficult to see the following.
\begin{lemma}\label{l:syndrome}
{\rm(i)} Any connected Cayley graph $\Cay(\FF_q^m,S)$
with the connecting set~$S$ invariant 
under the multiplication 
by nonzero elements of~$\FF_q$
(for $q\in\{2,3\}$, this condition is automatically satisfied) is a syndrome graph of a check matrix of some $[n,n-m,3]_q$ code~$C$, where $n=|S|/(q-1)$.

{\rm(ii)} Syndrome graphs of different check $m\times n$ matrices of the same $[n,n-m,3]_q$ code  or check $m\times n$ matrices of equivalent $[n,n-m,3]_q$ codes are isomorphic.
\end{lemma}

\begin{remark}
The inverse of (ii) is not true in general.
The reason is that an isomorphism between Cayley 
graphs is not necessarily 
related to an isomorphism
between underlying groups,
see e.g. \cite{Nowitz:1992:nonCI},
\cite{Li:2002:CI}, \cite{Morris:2024:nonCI}.
\end{remark}

So, if we want to find an $S$-CR code in $H(n,q)$, $q\in\{2,3\}$, we can try to search 
for an $S$-CR code in all connected Cayley graphs
of valency~$(q-1)n$ on~$\FF_q^m$, for some $m<n$.
To test all nonisomorphic such graphs, it is sufficient to run over syndrome graphs of all 
nonequivalent $[n,n-m, 3]_q$ codes.
For generating the list of all 
nonequivalent codes, we used the software {QextNewEdition}~\cite{QextNewEdition,BBK:2021}.
If an $S$-CR code is not found in any such graph,
this does not mean the nonexistence of an $S$-CR code in $H(n,q)$, but this does imply the nonexistence of an $S$-CR code that is the union of cosets of an $[n,n-m, 3]_q$ code.

A special variation of the approach above for strongly regular graphs is based on the connection (Lemma~\ref{l:srg} below) between the following two concepts.

\begin{definition}[\textbf{strongly regular graph}]
A connected graph on $v$ vertices is called \emph{strongly regular}, with parameters 
SRG$(v,k,\lambda,\mu)$, if every singleton
(one-vertex set) is a
$\left(\begin{smallmatrix}
0 & k & 0 \\
1 & \lambda & k{-}\lambda{-}1 \\
0 & \mu & k{-}\mu
\end{smallmatrix}\right)$-CR code.
\end{definition}

\begin{definition}[\textbf{two-weight code}]
Linear codes of dimension $k$ in $H(n,q)$ with exactly two distances $d_1$ and $d_2$ between distinct codewords are called \emph{two-weight codes}, or \emph{$[n,k,\{d_1,d_2\}]_q$ codes}.
\end{definition}

\begin{lemma}[\cite{CalKan:86:2weight}]
\label{l:srg}
In Lemma~\ref{l:syndrome}, 
the graph $\Cay(\FF_q^m,S)$ is strongly regular 
with eigenvalues
(of the adjacency matrix) 
$\theta_0>\theta_1>\theta_2$ 
if and only if 
the linear code dual to~$C$ is two-weight, with distances
$((q-1)n-\theta_i)/q$, $i=1,2$.
\end{lemma}

\subsection{Orthogonal arrays}
Orthogonal arrays are important discrete structures studied for both theoretical reasons and applications, see e.g. \cite{HSS:OA}.
Completely regular codes found in this work have interesting interpretations related to two known bounds on orthogonal arrays.
\begin{definition}
[\textbf{orthogonal arrays}]
A nonempty multiset $C$ of vertices of $H(n,q)$ is called an \emph{orthogonal array} of strength $t$, or
OA$(|C|,n,q,t)$
if every $H(n-t,q)$-subgraph of 
$H(n,q)$ contains exactly 
$\lambda =|C|/q^t$
elements of~$C$.
An orthogonal array is \emph{simple} if it is a usual set (without multiplicities
larger than~$1$). 
\end{definition}

Completely regular codes are orthogonal arrays whose strength can be found from the eigenspectrum of the quotient matrix.

\begin{lemma}[see e.g. {\cite[Corollary 1.5]{KroPot:Ch1}}]\label{l:oa}
Each cell of an equitable partition of $H(n,q)$ (in particular, any completely regular code) is an orthogonal array
of strength $t = \displaystyle \frac{(q-1)n-\theta}{q}-1$, where $\theta$ is the second largest eigenvalue of the quotient
matrix.
\end{lemma}

Of course, not all orthogonal arrays (even simple) are CR codes.
However, sometimes, from the parameters of an orthogonal array we can conclude that it is a completely regular code.
The following bound was proved in \cite{Friedman:92} for $q=2$ and in \cite{Bierbrauer:95} for the general case. In \cite[Theorem 15.36]{Bier2}, it is called the \emph{dual Plotkin bound}; however,
in some literature it is referred to as
the \emph{Bierbrauer--Friedman bound}.

\begin{lemma}[see e.g. {\cite[Section 1.4.3.1]{KroPot:Ch1}}]
\label{l:BF}
Any orthogonal array $C$ of strength $t$ in $H(n,q)$ satisfies
\begin{equation}\label{eq:BF}
    |C| \ge q^n \Big ( 1 - \frac{(q-1)n}{q(t+1)} \Big).
\end{equation}
Moreover, any orthogonal array attaining bound~\eqref{eq:BF} is simple, independent, and CR-$1$.

Conversely, any independent CR-$1$ set in $H(n,q)$ is an orthogonal array attaining bound~\eqref{eq:BF}, where $t$ is found from the spectrum of the quotient matrix as in Lemma~\ref{l:oa}.
\end{lemma}

Orthogonal arrays attaining bound~\eqref{eq:BF} can be constructed as binary codes or, more generally, 
as additive codes \cite[Theorem 4.8]{BKMTV}.
Binary ($q=2$) orthogonal arrays  attaining~\eqref{eq:BF} whose cardinality is not a power of~$2$
can be constructed by the construction 
of CR-$1$ codes in~\cite{FDF:PerfCol}. 
In the nonbinary case $q>2$, such orthogonal arrays (attaining~\eqref{eq:BF} and with cardinality that is not a power of the prime divisor of~$q$) have not been known before. 
We have found the first two examples, $\{22;5\}$-CR in $H(11,3)$ and $\{50;4\}$-CR in $H(25,3)$, in Sections~\ref{s:243} and~\ref{s:81}, respectively.
With constructions in Proposition~\ref{p:*s}(ii) and Corollary~\ref{c:spli} in the next section, these examples 
provide infinite series of independent CR-$1$ sets and orthogonal arrays attaining~\eqref{eq:BF} whose cardinality is not a power of~$3$.

Another bound is known only for the binary case.
We are going to discuss its hypothetical generalizations and, to this end, we rewrite it using the two following concepts.
\begin{definition}[\textbf{unbalanced set, relative strength}]
A set of vertices in $H(n,q)$ is called \emph{unbalanced}
if $|C|$ is not divisible by $q^{n-1}$.
For an OA$(N,n,2,t)$, the \emph{relative strength} is the value
$\tau = (t+1)/n$.
\end{definition}

\begin{lemma}[correlation-immunity bound \cite{FDF:CorrImmBound}]\label{l:ci}
 The relative strength of a simple unbalanced binary orthogonal array does not exceed $2/3$. Moreover simple unbalanced binary orthogonal arrays with relative strength $2/3$ are necessarily CR-$1$.
\end{lemma}
Based on the known constructions of CR-$1$, the following conjecture was proposed in \cite{BKMTV}, as a hypothetical generalization of Lemma~\ref{l:ci}.
\begin{conjecture}[\cite{BKMTV}]
    (i) The relative strength $\tau=(t+1)/n$ of a simple unbalanced orthogonal array OA$(N,n,q,t)$ does not exceed $q/(q+1)$.
    (ii) In particular, the relative strength $\tau=(b+c)/qn$ of a $\{b;c\}$-CR code where $b/(b+c)$ is not a multiple of $1/q$ (the unbalance condition) does not exceed $q/(q+1)$.
\end{conjecture}
As we will see in the summary table in Section~\ref{s:sum}, most of CR codes found in the current work disprove this conjecture even in its weak form (ii). The problem of an upper bound on the relative strength
of $q$-ary simple unbalanced orthogonal arrays (in particular, CR-$1$ codes) remains open.

\section{Constructions of CR codes}
\label{s:constr}

In this section, we show several known (see \cite{BKMTV} and \cite{KooKroMar:Ch7tables}) constructions of CR-$1$ codes and a new construction of CR-$2$ codes. The most important conclusion from these constructions is that every newly found CR code can yield an infinite family of CR codes. 
In particular, the examples of independent CR-$1$ sets
found in Sections~\ref{s:243}
and~\ref{s:81} 
(a $\{22;5\}$-CR code in $H(11,3)$ 
 and a $\{50;4\}$-CR code in $H(25,3)$,
 respectively)
lead to infinite $2$- and $3$-parameter
families of independent CR-$1$ and CR-$2$ sets,
respectively,
whose sizes are not powers of~$3$.

\begin{proposition}[see, e.g., the last paragraph of Section~1.3.4 in~\cite{KroPot:Ch1}]
\label{p:+t}
\label{p:*s}
  If there is an equitable partition of $H(n,q)$ with 
quotient matrix~$S$, then for all positive integer
$t$ and $s$, the following equitable partitions exist:
\begin{enumerate}
\item [\rm(i)] an equitable 
partition of $H(n+t,q)$ with 
quotient matrix $S+t(q-1)\mathrm{Id}$;
\item[\rm(ii)]
an equitable 
partition of $H(sn,q)$ with 
quotient matrix $sS$.
\end{enumerate}
In particular, if there is a CR code with intersection array $A$ in $H(n,q)$, then {\rm(i)} an $A$-CR code exists 
in $H(n+t,q)$ and {\rm(ii)} an $sA$-CR code exists 
in $H(sn,q)$.
\end{proposition}

\begin{proof}
We assume without loss of generality that 
$H(n,q)$ is defined on the vertex set 
$G^n$, where $(G,+)$ is an arbitrary Abelian group of order~$q$ (for example, the additive group of $\ZZ_q$ or $\FF_q$).

{\rm(i)}
If $(C_0,\ldots,C_{k-1})$ is an equitable $S$-partition
of $H(n,q)$, then
$(C'_0,\ldots,C'_{k-1})$,
where 
$$C'_i =  \{ (x_0,\ldots,x_{n+t-1})\in G^{n+t}:\ (x_0,\ldots,x_{n-1}) \in C_i \},$$
is an equitable $(S+t(q-1)\mathrm{Id})$-partition
of $H(n+t,q)$.

{\rm(ii)}
If $(C_0,\ldots,C_{k-1})$ is an equitable $S$-partition
of $H(n,q)$, then
$(C''_0,\ldots,C''_{k-1})$,
where 
$$C''_i =  \{ (x_0,\ldots,x_{s-1})\in (G^n)^s:\ x_0+\ldots+x_{s-1} \in C_i \},$$
is an equitable $sS$-partition
of $H(sn,q)$. Indeed, if $x\in C_i$ has a neighbor
$x+e$ in $C_j$,
then every $(x_{0},\ldots,x_{s-1})\in C''_i$ such that
$x_0+\ldots+x_{s-1}=x$ has the following $s$ corresponding neighbors in~$C''_j$:
\begin{equation}\label{eq:yl}
   y^{(l)}=(x_0,\ldots,x_{l-1},x_l+e,x_{l+1},\ldots,x_{s-1}),
   \qquad
l=0,\ldots,s-1.
\end{equation}
\end{proof}

\begin{proposition}[splitting {\cite[Th.\,4.13]{BKMTV}}]\label{p:split}
  If there is an equitable $[[a,b],[c,d]]$-partition $(C_0,C_1)$, $a\le c$, 
  in $H(n,q)$ ($a+b=c+d=(q-1)n$),
  where $q$ is a power of a prime,
  then in $H(qn+c-a,q)$
  there is an equitable $2$-partition 
  with quotient matrix 
$\begin{pmatrix}
    a+(i{-}1)c & (q{-}i)c + qb \\
    ic & qb+a+(q{-}i{-}1)c
\end{pmatrix}
$,
$i=1,\ldots,q-1$.
\end{proposition}
\begin{proof}
We describe a four-step construction for the case when $q$ is prime; the general case is similar.

(i) We start with constructing an equitable $\begin{pmatrix}qa&qb\\qc&qd\end{pmatrix}$-partition $(C'_0,C'_1)$ as in the proof of Proposition~\ref{p:*s}(ii), where $s=q$.

(ii)
Then, we split $C'_0$ into $q$ subcells
$C^{\prime 0}_0$, \ldots, $C^{\prime q-1}_0$,
according to the value of
\begin{equation}\label{eq:al}
\alpha(y_0,\ldots,y_{q-1})=0\cdot|y_0|+ 1\cdot|y_1|+ \ldots + (q-1)\cdot |y_{q-1}|
\end{equation}
for each $(y_0,\ldots,y_{q-1})$ in~$C'_0$,
where $|y_i|$ denotes the sum of all $n$ components of the $n$-tuple~$y_i$.
It is not difficult to see that all $q$
vectors $y^{(l)}$, defined in~\eqref{eq:yl}, have distinct
values of~$\alpha(y^{(l)})$.
It follows that any vertex of $H(qn,q)$ has equal numbers of neighbors in $C^{\prime 0}_0$, \ldots, $C^{\prime q-1}_0$, and
$(C^{\prime 0}_0, \ldots, C^{\prime q-1}_0,C^{\prime}_1)$ is an equitable partition with
quotient matrix 
\begin{equation}\label{eq:orQM}
    \begin{pmatrix}
    a & \cdots & a & qb \\
    \vdots & \ddots & \vdots & \vdots \\
    a & \cdots & a & qb \\
    c & \cdots & c & qd
\end{pmatrix}.
\end{equation}

(iii)
Next, we construct a partition 
$(D^{ 0}_0, \ldots, D^{ q-1}_0,D_1)$ of $H(qn+1,q)$
as follows: 
$$ D^{ i}_0 =
\{(x,y) \in \ZZ_q^{qn+1}:\ 
y \in \ZZ_q,\ x\in C^{\prime i+y}_0\},$$
$$ D_1 =
\{(x,y) \in \ZZ_q^{qn+1}:\ 
y \in \ZZ_q,\ x\in C'_1\}.$$
It is an equitable partition with quotient matrix
\begin{equation}\label{eq:QM1}
\begin{pmatrix}
    a & a{+}1 & \cdots & a{+}1 & qb \\
    a{+}1 & a & \cdots & a{+}1 & qb \\
    \vdots & \vdots & \ddots & \vdots & \vdots \\
    a{+}1 & a{+}1 & \cdots & a & qb \\
    c & c &  \cdots & c & qd+q{-}1
\end{pmatrix}.
\end{equation}
Indeed, a vertex $x$ in $D^{ i}_0$ 
has $q-1$ neighbors differing with~$x$ in the
last, ``new'', coordinate.
By the definition of $D^{ i}_0$, these
neighbors belong to $q-1$ distinct cells
$D^{j}_0$, $j \in \ZZ_q\backslash\{i\}$,
which corresponds to the $q-1$ ``$+1$s''
in the $i$th row of the quotient matrix above.
The other $qn(q-1)$ neighbors,
differing with~$x$ in ``old'' coordinates,
are colored 
with respect to the $i$th row of the quotient matrix~\eqref{eq:QM1}. For the last row, the argument is similar.

Applying step (iii) $c-a$ times, we obtain
an equitable partition $(E^{0}_0, \ldots, E^{ q-1}_0,E_1)$
of $H(nq{+}c{-}a,q)$
with quotient matrix
$$
\begin{pmatrix}
    a & c & \cdots & c & qb \\
    c & a & \cdots & c & qb \\
    \vdots &\vdots &  \ddots & \vdots & \vdots \\
    c & c & \cdots & a & qb \\
    c & c &  \cdots & c & qd+(q{-}1)(c{-}a)
\end{pmatrix}.
$$

(iv) Now, because of equal coefficients 
in cells of the quotient matrix,
we can get an equitable $2$-partition
by unifying the first $i$ cells and
the remaining $q-i+1$ cells, $i=1,...,q-1$.
The resulting quotient matrix is
$$
\begin{pmatrix}
    a+(i{-}1)c & (q{-}i)c + qb \\
    ic & qb+a+(q{-}i{-}1)c
\end{pmatrix}.
$$
\end{proof}

\begin{remark}
The general splitting construction
(see \cite[Sect.\,3]{FDF:PerfCol} for $q=2$ and \cite[Th.\,4.14]{BKMTV} in general)
is more complicated and can sometimes get a better result
if the first cell $C_0$ of the original partition can be split into $s$-faces
(sets inducing $H(s,q)$-subgraphs).
Since we could not apply this improvement
in the current work, we show in 
Proposition~\ref{p:split} only the simplified version of the construction,
and give a simplified proof.
\end{remark}
The following special case ($a=0$, $i=1$) of Proposition~\ref{p:split} is significant 
in the context of optimal orthogonal arrays,
see Lemma~\ref{l:BF}.
\begin{corollary}\label{c:spli}
If there is an independent
$\{b;c\}$-CR set in $H(n,q)$,
$b=(q{-}1)n$,
then there is an independent
$\{(q{-}1)c{+}qb;c\}$-CR set in $H(qn{+}c,q)$.
\end{corollary}

Now, we are going to make a new (at least,
for $q\ge3$) observation, 
which is important in the context
of our current study.
Namely, that from independent CR-$1$ sets
we can obtain independent CR-$2$ sets.

\begin{theorem}\label{th:1to2}
If there is a $\{b;c\}$-CR code $C$ in $H(n,q)$,
where $b=n(q-1)$ (i.e., $C$ is an independent set), then for every integer $j\ge 1$ and every $i\in\{1,\ldots,q^j-1\}$ there exists a
$\{q^j b,(q^j-i)c;ic,q^j b\}$-CR code 
in $H(q^j n,q)$.
\end{theorem}
\begin{proof}
Let us consider the equitable partition
$(C^{\prime 0}_0, \ldots, C^{\prime q-1}_0,C^{\prime}_1)$
with quotient matrix~\eqref{eq:orQM},
where $a=0$,
constructed after step~(ii) in the proof of Proposition~\ref{p:split}. 
To this partition,
we can again apply steps~(i) and~(ii),
first multiplying the quotient matrix by~$q$,
and then splitting all cells except the last one 
into $q$ subcells according to the value
of~$\alpha$ defined in~\eqref{eq:al}.
Applying, in such a way, steps~(i--ii) $j$ times, we obtain an equitable $(q^j+1)$-partition
$(D^{0}_0, \ldots, D^{ q^j-1}_0,D^{}_1)$
with quotient matrix
\begin{equation*}
    \begin{pmatrix}
    0 & \cdots & 0 & q^j b \\
    \vdots & \ddots & \vdots & \vdots \\
    0 & \cdots & 0 & q^j b \\
    c & \cdots & c & q^j d
\end{pmatrix}.
\end{equation*}

From the quotient matrix,
one can directly check that the $3$-partitions
$(D,E,D_1)$, where
$D=D^{ 0}_0+\ldots+D^{ i-1}_0$,
$E=D^{ i}_0+\ldots+D^{ q^j-1}_0$,
and hence $(D,D_1,E)$
are equitable with quotient matrices
$$
\begin{pmatrix}
0&0&q^j b \\
0&0&q^j b \\
ic&(q^j -i)c&q^j d
\end{pmatrix}
\quad
\mbox{and}\quad
\begin{pmatrix}
0&q^j b&0 \\
ic&q^j d&(q^j-i)c\\
0&q^j b&0 \\
\end{pmatrix},
$$
respectively.
Therefore, $D$ is a required $\{q^j b,(q^j-i)c;ic,q^j b\}$-CR code.
\end{proof}

Using Proposition~\ref{p:*s}(ii), Corollary~\ref{c:spli},
and Theorem~\ref{th:1to2},
each new independent CR-$1$ set
yields an infinite 
family of CR-$2$ codes with a
three-parameter family of intersection arrays.
\begin{example}
Let us start with the $\{22;5\}$-CR in $H(11,3)$ found in Section~\ref{s:243}.
Applying Corollary~\ref{c:spli} $l$ times, we obtain a $\{3^{3+l}-5;5\}$-CR code in $H(\frac{3^{3+l}-5}{2},3)$.
Then, by Theorem~\ref{th:1to2},
we have a $\{ 3^{3+l+j}-5\cdot 3^j,5(3^j-i);5i,3^{3+l+j}-5\cdot 3^j \}$-CR code,
with quotient matrix
$$
\begin{pmatrix}
0 & 3^{3+l+j}-5\cdot 3^j & 0 \\
5i & 3^{3+l+j}-10\cdot 3^j & 5\cdot 3^j-5i\\
0 & 3^{3+l+j}-5\cdot 3^j & 0 \\
\end{pmatrix},
\quad \mbox{in } H\left(\frac{3^{3+l+j}-5\cdot 3^j}{2},3\right)
$$
for every integer $l\ge 0$, $j\ge 1$, and $i\in\{1,\ldots,3^j-1\}$. 
By Proposition~\ref{p:*s}(ii), we can multiply this quotient matrix by any positive integer constant~$s$;
by Proposition~\ref{p:+t}(i), we can add 
any nonnegative even constant $2t$ to the diagonal,
which does not change the intersection array.
So, we have a $4$-parameter class of CR-$2$ codes,
while the intersection array is only $3$-parameter
(the values $s3^j$ and $si$ are naturally seen as two parameters, not three).
The first case is $l=t=0$, $s=j=i=1$,
a $\{ 66,10;5,66 \}$-CR code in $H(33,3)$.
\end{example}

\section{CR codes in Cayley graphs}\label{s:CRinCay}
In this section, we report the results of a computational search of CR codes in Cayley graphs that are covered by Hamming graphs $H(n,q)$, $q=2,3$. We were focused on finding new parameters marked as putative in the tables
of small values~\cite{KooKroMar:Ch7tables}, 
so the strategy was as follows.

At first, we checked small strongly regular Cayley graphs on $\FF_q^k$ that can potentially have CR codes
with required parameters (``can potentially have'' means, for a given graph and a given quotient matrix, that the cardinality of the corresponding CR code is an integer and the eigenvalues of the matrix are eigenvalues of the graph).
For each such graph, we found all parameters of CR codes, of covering radius~$1$ and~$2$. We list all parameters of all CR codes in one of Sections \ref{s:81-20}--\ref{s:28} below, corresponding to the graph, even when they do not lead to CR codes with new parameters in Hamming graphs.

Next, we considered all small Cayley graphs on $\FF_q^k$ (we checked
Cayley graphs on $\FF_3^4$ up to degree~$54$ and
Cayley graphs on $\FF_2^6$ up to degree~$21$)
and searched there for CR codes with the remaining putative parameters.
Apart from the theoretical construction 
in Section~\ref{s:Hadamard} 
(the new parameters are outside of the small-parameter table),
we did not find binary codes with new parameters.
For ternary codes, however, the search was much more productive: 
our findings solve    
$11$ intersection arrays $\{b;c\}$ of CR-$1$ codes,
establishing the smallest~$n$ for which $\{b;c\}$-CR codes exist; for $5$ more arrays, the upper bound
on such~$n$ has been improved (see Table~\ref{tab:ias} in Section~\ref{s:sum} and compare with \cite[Appendix A1]{BKMTV});
a CR-$2$ code with completely new intersection array has been found.

All Cayley graphs on $\FF_q^k$ with required $q$, $k$, degree, and, if necessary, parameters of SRG, were found with software {QextNewEdition}~\cite{QextNewEdition,BBK:2021}, as discussed in Section~\ref{s:LinCay}.
For a relatively small graph, the existence of a CR code with given quotient matrix~$S$ can be treated as the existence of a solution of an integer linear program, ILP, and solved with the corresponding software (actually, one needs to find a rectangular $\{0,1\}$-matrix $F$ satisfying $AF=FS$, where $A$ is the adjacency matrix of the graph, see more details in~\cite[1.3.6]{KroPot:Ch1});
we used GAMS \cite{GAMS} with the solver CPLEX.

\subsection{\texorpdfstring{$\{66,15;1,54\}$-CR code in $H(66,2)$}{\{66,15;1,54\}-CR code in H(66,2)}}\label{s:Hadamard}
We start with a simple observation,
which illustrates the general technique
of constructing a CR code 
by covering a Cayley graph
and
leads to a CR code
with minimum distance~$3$. 
To the best of our knowledge, this code is new and 
it is the second known CR code in $H(n,2)$
whose minimum distance is larger than~$2$
and the size is not a power of~$2$.
The first one is the Hadamard code in $H(12,2)$, with minimum distance~$6$ and cardinality~$24$.
Its quotient matrix is the matrix $S$ below:
\begin{equation}
S = \begin{pmatrix}
    0 & 12 &  0 &  0 & 0  \\
    1 &  0 & 11 &  0 & 0  \\
    0 &  2 &  0 & 10 & 0  \\
    0 &  0 &  9 &  0 & 3  \\
    0 &  0 &  0 & 12 & 0  
\end{pmatrix}; 
\qquad
\frac{S^2-12{I}}{2} =
\begin{pmatrix}
 \mathbf{0} & 0 &\mathbf{66} & 0 & \mathbf{0} \\
 0 &11 & 0 &55 & 0 \\
 \mathbf{1} & 0 &\mathbf{50} & 0 &\mathbf{15} \\
 0 & 9 & 0 &57 & 0 \\
 \mathbf{0} & 0 &\mathbf{54} & 0 &\mathbf{12} 
\end{pmatrix}.
\end{equation}
The corresponding equitable $S$-partition of
$H(12,2)$ is also an equitable $\frac{S^2-12{I}}{2}$-partition of the distance-$2$ graph of $H(12,2)$, where the edges are the pairs of vertices at Hamming distance~$2$ from each other.
(In general, any equitable partition of a distance-regular graph~$\Gamma$ is an equitable partition of the distance-$i$ graph of~$\Gamma$, see, e.g., \cite[Section 1.3.3]{KroPot:Ch1}.)
This graph has two connected components,
often denoted $\frac12 H(12,2)$ (a \emph{halved hypercube}), and we have two equitable partitions of $\frac12 H(12,2)$,
with quotient matrices 
\begin{equation}
S' = \begin{pmatrix}
    0 & 66 &  0    \\
    1 &  50 & 15   \\
    0 &  54 &  12   
\end{pmatrix},
\qquad
S'' = \begin{pmatrix}
    11 &  55    \\
    9 &  57   
\end{pmatrix}.
\end{equation}
The first equitable partition corresponds to the distance partition in $\frac12 H(12,2)$ with respect to the original Hadamard code~$C$, which lies in only one of the two instances of $\frac12 H(12,2)$. It follows that $C$ is a $\{66,15;1,54\}$-CR code in $\frac12 H(12,2)$. On the other hand, $\frac12 H(12,2)$ is a degree-$66$ Cayley graph on the $11$-dimensional subspace of $\FF_2^{12}$ consisting of all binary $12$-tuples with even number of $1$s (the connecting set consists of the $66$ such tuples with two ones). As a preimage under covering, $C$ induces a $\{66,15;1,54\}$-CR code in $H(66,2)$. From the left column of the quotient matrix, we see that there are no codewords at mutual distance~$1$ or~$2$ from each other. The size of the new code is $24\cdot 2^{66-11} = 864691128455135232$.

\subsection{\texorpdfstring{A $\{21,4;2,21\}$-CR code in $H(11,3)$}{A \{21,4;2,21\}-CR code in H(11,3)}}
\label{s:21,4,2,21}
A $\{21,4;2,21\}$-CR code is one of the most 
interesting findings in the current work.
This intersection array is completely new:
no $\{21,4;2,21\}$-CR codes were known before in any Hamming graphs. The inner degree of the new code is only~$1$, that is, the codewords induce a matching in $H(11,3)$.
Moreover, the code is related to an equitable $4$-partition (see the quotient matrix \eqref{eq:C012} below), where three of the cells are $\{21,4;2,21\}$-CR codes.
Motivated by hypothetical generalizations, 
we pay more attention to this code and,
after the description of computational results, give a theoretical construction of the code in a separate subsection.

There are $401$ nonisomorphic connected Cayley graphs of degree $2\cdot 11$ 
on $\FF_3^4$. Having checked all of them, 
we have discovered the following finding.
One (and only one) of these graphs, the coset graph $G(H)$
of the code $C(H)$ with check matrix 
$$
H = \left(\begin{array}{c@{\ \ }c@{\ \ }c@{\ \ }cc@{\ \ }c@{\ \ }c@{\ \ }cc@{\ \ }c@{\ \ }c}
1&0&0&0&0&1&1&1&1&2&2 \\
0&1&0&0&1&0&1&1&2&1&2 \\
0&0&1&0&1&1&0&1&2&2&1 \\
0&0&0&1&1&1&1&0&1&1&1
\end{array}\right),
$$
has a $\{21,4;2,21\}$-CR code
$C=\{0000, 2211, 1012, 0121, 1220, 2102\}$,
with quotient matrix 
$$\begin{pmatrix}
    1&21&0\\2&16&4\\0&21&1
\end{pmatrix}.$$
The graph has diameter $2$, eigenspectrum 
$\{22^1,10^2,4^{12},1^{44},-5^{14},-8^8\}$;
the stabilizer of a vertex~$o$ has order $48$ and
orbits of sizes 
$1$, $8$, $8$, $6$, $12$, $8$, $24$, $12$, $2$
($8$, $8$, $6$ for neighbors of~$o$, in particular,
the code $C(H)$ has three coordinate orbits,
$\{0,1,2,3\}$, $\{4,5,6,7\}$, and $\{8,9,10\}$,
and the graph has three corresponding edge orbits).
$\Aut(C)$ has order $24$, 
acts transitively on $C$ and $C^{(2)}$ and divides $C^{(1)}$ into $4$ orbits of sizes $3$, $12$, $24$, $24$.
$C^{(2)}$ is uniquely divided into two other $\{21,4;2,21\}$-CR codes $C'$ and $C''$; the partition
$(C,C',C'',C^{(1)})$ is hence equitable with quotient matrix
\begin{equation}
\begin{pmatrix}\label{eq:C012}
    1&0&0&21\\
    0&1&0&21\\
    0&0&1&21\\
    2&2&2&16\\
\end{pmatrix}.
\end{equation}
We note that $(C\cup C^{(2)},C^{(1)})$ is a $\{21;6\}$-CR code in $G(H)$, which induces a new $\{21;6\}$-CR code in $H(11,3)$.
Another $\{21;6\}$-CR code is found in Section~\ref{s:243}.

A more detailed analysis of the structure of the new $\{21,4;2,21\}$-CR code in $H(11,3)$ 
leads to a theoretical construction of it, which is considered in the rest of this section.

\subsubsection{Theoretical construction}\label{s:the}
Let $D = K + R$, where $K$ is a linear code of size~$3^8$,
the linear combination of the rows of the matrix~$A$ below, and $R$ is a code of size~$6$,
consisting of the rows of the matrix~$B$.
\begin{align} \label{eq:As}
    A &= \left(\begin{array}{ccc|cccccccc}
        1 & 1 & 1 & 1&0&0&0&0&0&0&0  \\
        1 & 1 & 2 & 0&1&0&0&0&0&0&0  \\
        1 & 2 & 1 & 0&0&1&0&0&0&0&0  \\
        1 & 2 & 2 & 0&0&0&1&0&0&0&0  \\
        2 & 1 & 1 & 0&0&0&0&1&0&0&0  \\
        2 & 1 & 2 & 0&0&0&0&0&1&0&0  \\
        2 & 2 & 1 & 0&0&0&0&0&0&1&0  \\
        2 & 2 & 2 & 0&0&0&0&0&0&0&1  \\ 
    \end{array}\right),
    \quad\sigma:
    \left[\begin{array}{c}
        0\\1\\1\\0\\1\\0\\0\\1
    \end{array}\right],
\\ \label{eq:Bs}
    B &= \left(\begin{array}{ccc|cccccccc}
        0 & 0 & 1 & 0&0&0&0&0&0&0&0  \\
        0 & 0 & 2 & 0&0&0&0&0&0&0&0  \\
        0 & 1 & 0 & 0&0&0&0&0&0&0&0  \\
        0 & 2 & 0 & 0&0&0&0&0&0&0&0  \\
        1 & 0 & 0 & 0&0&0&0&0&0&0&0  \\
        2 & 0 & 0 & 0&0&0&0&0&0&0&0  \\
    \end{array}\right),
    \quad\sigma:
    \left[\begin{array}{c}
        0\\0\\1\\1\\2\\2
    \end{array}\right].
\end{align}

\begin{proposition}\label{p:12-6}
    We claim that $D$ is a $\{21;6\}$-CR code in~$H(11,3)$.
\end{proposition}
\begin{proof}
To see this, 
we need to check that every codeword~$x$
has exactly one codeword neighbor 
and every noncodeword~$x$ has exactly 
$6$ codeword neighbors.

We first note that for any vector~$x$, a codeword or not,
there is a translation vector~$y=y(x)$ in~$K$
such that
$y+x$ has zeros in the last $8$ positions
(explicitly, $y(x) = - x'' {\cdot} A\in K$, where $x''$ is the length-$8$ suffix of~$x$).
Since the translation by~$y\in K$ is an automorphism of~$D$, i.e., $y+D=D$, both $x$ and $y+x$ have the same number of neighbors in~$D$ and, moreover, they are codewords or noncodewords of~$D$ simultaneously.
So, in the following considerations, we assume without loss of generality that $x$ ends with $8$ zeros.

If $c=a+b$, where $\bar 0 \ne a\in K$ and $b\in R$,
then either $\pm a$ is a row of~$A$ and $c$ has at least three nonzero symbols or $c$ has at least two nonzero symbols in the last~$8$ positions. In each case, $c$ cannot be a neighbor of $x$ if $x$ has at most one nonzero symbol and ends with
$8$ zeros. Immediately, we get (i) and (ii) below.

(i) If $x = \bar 0\not\in D$, then $x$ has exactly $6$ neighbors in~$D$, namely, the elements of~$R$.

(ii) If $x = (x',\bar 0)\in D$,
$x'\in\{001,002,010,020,100,200\}$,
then $x$ has exactly one neighbor in~$D$,
namely, $-x$.

(iii) Let $x = 0\alpha\beta\bar 0$, $\alpha0\beta\bar 0$, or $\alpha\beta0\bar 0$, where $\alpha,\beta\in\{1,2\}$.
Assume $x = 011\bar 0\not\in D$; the other cases are similar.
Then $x$ has exactly $2$ neighbors in~$R$,
namely $001\bar 0$ and $010\bar 0$.
The other codeword neighbors $c$ 
of~$x$ have the form $c=a+b$, where $b\in R$ and $\bar0 \ne a\in K$. Since $a$ can have only one nonzero symbol in the last~$8$ positions, either $a$ or $-a$ is a row of the matrix~$A$. It follows that $a$ starts with three nonzeros, $b\in \{100\bar 0,200\bar 0\}$, and $a\in\{211\,00001000,$ $211\,00020000\}$ or $a\in\{111\,10000000,$ $111\,00000002\}$, respectively. We have found that $x$ has $2+4$ codeword neighbors in total.

(iv) Finally, let $x = \alpha\beta\gamma\bar 0$, $\alpha,\beta,\gamma\in\{1,2\}$.
Assume $\alpha\beta\gamma = 111\not\in D$
without loss of generality.
Consider a neighbor~$c$ of~$x$ 
of the form $c=a+b$ where $b\in R$ and $a\in K$. 
Obviously $a\ne 0$ and hence $a$ or $-a$
is a row of~$A$,
starting with three nonzeros.
It follows that
$b\in \{002\bar 0,
020\bar 0,200\bar 0\}$ 
and $a$ belongs to 
$\{112\,01000000, 112\,00000020\}$, 
$\{121\,00100000, 121\,00000200\}$, or 
$\{211\,00001000, 211\,00020000\}$,
respectively.
We see that $x$ has $6$ neighbors in~$D$.
\end{proof}

Next, we are going to split $D$ into three subcodes, to form an equitable partition with quotient matrix~\eqref{eq:C012}.
On the codewords~$x$ of~$D$,
we define an $\FF_3$-valued function~$\sigma(x)$,
whose values on $R$ and the basis of~$K$
are shown in~\eqref{eq:Bs} and~\eqref{eq:As},
respectively, being expanded linearly to~$K$ and to~$D$ by the rule $\sigma(a+b)=\sigma(a)+\sigma(b)$, $a\in K$, $b\in R$. Define $C$, $C'$, $C''$ as the set of codewords~$x$ of~$D$ with $\sigma(x)=0$, $1$, $2$, respectively, and $C^{(1)}$ as the complement of~$D$.
\begin{proposition}\label{p:C012}
    The partition $(C,C',C'',C^{(1)})$ is equitable with quotient matrix~\eqref{eq:C012}.
\end{proposition}
\begin{proof}
We will refer to the value~$\sigma(x)$
as the \emph{color} of~$x$.
We will also define
the parity of a ternary word~$x$ as follows:
$x$ is even (odd)
if it has even (odd, respectively)
number of $1$s. In particular, for a row $a$ of the matrix $A$, we have
$$ \sigma(a) = \begin{cases} 0 & \text{if $a$ is even,} \\ 1 & \text{if $a$ is odd;}
\end{cases}\qquad
\sigma(-a) = \begin{cases} 0 & \text{if $-a$ is even,} \\ 2 & \text{if $-a$ is odd.}
\end{cases}
$$

We will revisit the proof of Proposition~\ref{p:12-6} and check that every 
codeword~$x$ has only a codeword neighbor 
of the same color and every noncodeword 
has exactly two codeword neighbors of each color. Again, because of the linearity of $K$ and $\sigma$ on~$K$, we can assume without loss of generality that $x$ ends with $8$ zeros.

If $x$ has at most one nonzero symbol, then the claim is obvious from (i-ii) in the proof of Proposition~\ref{p:12-6} and from the values of~$\sigma$ on~$R$, see~\eqref{eq:Bs}.

If $x=011\,\bar 0$, then we see from (iii) that $x$ has exactly two neighbors in $R$, of colors $0$ and~$1$. The other $4$ codeword neighbors have the form
$a+b$, where $\sigma(b)=2$ and $\sigma(a)$ possesses the values $1$, $0$, $0$, $2$. So,
these codewords have the colors $0$, $2$, $2$, $1$, respectively, and the claim is true. The arguments are the same for all $x$ with two nonzeros in the first three positions.

If $x=\alpha\beta\gamma\,\bar 0$, $\alpha,\beta,\gamma\in\{1,2\}$, then  we see from (iv)
that the $6$ codeword neighbors of $x$
have the form $a+b$, where 
$b\in\{-\alpha00\,\bar 0,-0\beta0\,\bar 0,-00\gamma\,\bar 0\}$
 and for each of the three values of~$b$,
there are two values of~$a$, one of which is a row of~$A$ and the other is opposite to a row of~$A$.
We have 
$\sigma(-\alpha00\,\bar 0)=2$,
$\sigma(-0\beta0\,\bar 0)=1$,
$\sigma(-00\gamma\,\bar 0)=0$.
The values of~$a$ that are rows of~$A$ have the same parity as~$x$, because they differ with~$x$ in one of the first $3$ positions (which contributes to the parity) and also have~$1$ in one of the last $8$ positions (which contributes to the parity);
hence, for all these three values of~$a$, the value of $\sigma(a)$ is the same, say $\sigma'$ (explicitly $\sigma'=0$ for even~$x$ and $\sigma'=1$ for odd~$x$, but we do not need the concrete value).
The other three values of~$a$ have the parity different from 
the parity of~$x$, because they differ with~$x$ in one of the first $3$ positions (which contributes to the parity) and have~$2$ in one of the last $8$ positions (which does \emph{not} contribute to the parity);
hence, they have the same value of $\sigma(a)$, say $\sigma''$ ($\sigma''=2$ for even~$x$ and $\sigma''=0$ for odd~$x$).
We conclude that for the six values of $a+b$, 
we have $\sigma(a+b)=2+\sigma'$, $2+\sigma''$, $1+\sigma'$, $1+\sigma''$, $0+\sigma'$, $0+\sigma''$. Independently of the concrete values of $\sigma'$ and $\sigma''$, we have two $0$s,
two $1$s, and two $2$s.
\end{proof}

\begin{corollary}\label{c:C0}
    The code $C$, constructed as above, is a $\{21,4;2,21\}$-CR code in $H(11,3)$.
\end{corollary}

\subsection{\texorpdfstring{A $\{10;8\}$-CR code in $H(7,3)$}{A \{10;8\}-CR code in H(7,3)}}
\label{s:10,8}
The first open question about the parameters of CR-$1$ codes in $H(n,3)$ (see \cite[Appendix A1]{BKMTV}) was the existence of a $\{10;8\}$-CR code in $H(7,3)$. In $H(8,3)$, $\{10;8\}$-CR codes were known to exist and obtained by the doubling construction (Proposition~\ref{p:*s}(ii), $s=2$) from $\{5,4\}$-CR codes in $H(4,3)$.
On the other hand, $\{10;8\}$-CR codes in $H(6,3)$ cannot exist by \cite[Theorem 3.3]{BKMTV}
(briefly, the orthogonal array related to an $\{10;8\}$-CR code in $H(6,3)$ would have the noninteger parameter $\lambda=4/3$).

There are no Cayley SRG over $\FF_3$ with valency~$14$. So, 
we searched 
for a $\{10;8\}$-CR code
in all Cayley graphs on $\FF_3^4$
(all such graphs were enumerated by enumerating all linear
$[7,3,3]_3$ codes with
{QextNewEdition}~\cite{QextNewEdition,BBK:2021}).
A CR code with quotient matrix $\left(\begin{smallmatrix}4&10\\8&6\end{smallmatrix}\right)$
was found in the syndrome graph (of order $81$ and degree $14$) of the matrix
$$
\left(\begin{array}{ccccccc}
1&0&0&0&1&1&1 \\
0&1&0&0&1&0&0 \\
0&0&1&0&0&1&0 \\
0&0&0&1&0&0&1
\end{array}\right).
$$

\subsection{\texorpdfstring{A $\{28;8\}$-CR code in $H(15,3)$}{A \{28;8\}-CR code in H(15,3)}}\label{s:28-8}

A $\left(\begin{smallmatrix}2&28\\8&22\end{smallmatrix}\right)$-CR code
exists in the syndrome graphs $G$ of the matrices:
$$
\left(\begin{array}{ccccccccccccccc}
0&1&1&1&1&1&1&1&1&1&1&1&0&0&0\\
1&0&0&0&0&1&1&1&1&2&2&0&1&0&0\\
0&0&0&1&1&0&0&1&1&0&2&0&0&1&0\\
1&1&2&0&1&0&2&0&1&1&2&0&0&0&1\\
\end{array}\right)
$$
$$
\left(\begin{array}{ccccccccccccccc}
1&1&1&1&1&1&1&1&1&1&1&1&0&0&0\\
0&0&0&0&1&1&1&1&2&2&2&0&1&0&0\\
0&1&1&2&0&0&2&2&0&1&1&0&0&1&0\\
2&0&1&0&0&1&0&1&0&0&1&0&0&0&1\\
\end{array}\right)
$$
The inner degree of a $\{28;8\}$-CR code is $2$; 
so, it induces the union of cycles. The lengths
of cycles are $6$, $12$ or $6$, $4$, $4$, $4$
for different found codes (the search was not exhaustive);
however, the lengths of cycles for codes induced in $H(15,3)$
can be larger.

Previously, $\{28;8\}$-CR codes were known in $H(16,3)$,
from a $\{7;2\}$-CR code in $H(4,3)$ (the union of two $1$-perfect codes). It is not yet clear whether there is a $\{28;8\}$-CR code (which must be an independent set) in $H(14,3)$.

\subsection{\texorpdfstring{$\{35;10\}$- and $\{25;20\}$-CR codes in $H(19,3)$}{\{35;10\}- and \{25;20\}-CR codes in H(19,3)}}\label{s:35-10}
$\left(\begin{smallmatrix}3&35\\10&28\end{smallmatrix}\right)$-CR 
and
$\left(\begin{smallmatrix}13&25\\20&18\end{smallmatrix}\right)$-CR codes
exist in the syndrome graph $G$ (of order $81$ and degree $38$) of the following matrix:
$$
\left(\begin{array}{c@{\ \ }c@{\ \ }c@{\ \ }c@{\ \ }c@{\ \ }c@{\ \ }c@{\ \ }c@{\ \ }c@{\ \ }c@{\ \ }c@{\ \ }cc@{\ \ }cc@{\ \ }c@{\ \ }cc@{\ \ }c}
0&0&0&0&1&1&1&1&1&1&1&1 &0&0 &1&1 &0&0&0\\
1&1&1&1&0&0&0&0&1&1&1&1 &0&0 &1&0 &1&0&0\\
0&0&1&2&0&0&1&2&0&0&1&2 &1&1 &0&0 &0&1&0\\
1&2&0&0&1&2&0&0&1&2&0&0 &1&2 &0&0 &0&0&1
\end{array}\right).
$$
The graph has $81\cdot 2^4\cdot 72$ automorphisms, $4$ edge orbits
(the corresponding groups of columns of the matrix have sizes $12$, $2$, $3$, $2$).

The subgraph induced by the found $\{35;10\}$-CR code has degree $3$ (as we see from the quotient matrix) and connected components of sizes 
$6$, $4$, $4$, $4$ (so, the last three are cliques).
The code has $3456$ automorphisms and is divided into $2$ orbits of sizes $6$ and $12$ (non-code orbits have sizes $24$, $24$, $12$, $3$).

Previously, $\{35;10\}$- and $\{25;20\}$- CR codes were known in $H(20,3)$,
from $\{7;2\}$-CR  and $\{5;4\}$-CR codes in $H(4,3)$ (the union of two and four $1$-perfect codes, respectively). It is not yet clear whether there is a $\{35;10\}$-CR code (which must be the union of independent edges) in $H(18,3)$ or a $\{25;20\}$-CR code in $H(n,3)$, $n=16,17,18$.

\subsection{SRG(81,20,1,6) and SRG(81,60,45,42)}
\label{s:81-20}
There is only one two-weight
$[10,4,\{6,9\}]_3$ code, up to equivalence (doubly-shortened Golay $[12,6,6]_3$ code), with generator matrix
$$ H = 
\begin{pmatrix}
      0&1&1&1&1&1&1&0&0&0 \\
      1&0&1&1&2&2&0&1&0&0 \\
      1&1&0&2&1&2&0&0&1&0 \\
      1&1&2&0&2&1&0&0&0&1 
\end{pmatrix}.
$$
By Lemma~\ref{l:srg}, the coset graph $G(H)$ of the code dual to this two-weight code is strongly regular.
It is an SRG$(81,20,1,6)$ (the Brouwer--Haemers graph \cite[10.28]{BvM:SRG}), and it has CR codes with the following intersection arrays (here and below, some parameters are underlined, which indicates that they contribute to the small-parameter tables of CR codes in $H(n,3)$):\\
eigenvalue $2$:
$\{16;2\}$ \cite{BvM:SRG}, $\{14;4\}$, $\{12;6\}$ \cite{BvM:SRG}, $\{10;8\}$;\\
eigenvalue $-7$:
$\{18;9\}$ \cite{BvM:SRG}, \underline{$\{15;12\}$} \cite{BvM:SRG};\\
covering radius~$2$:
$\{20, 18; 1, 6\}$, 
$\{18, 8; 1, 18\}$,
$\{18, 7; 2, 18\}$,
$\{18, 6; 3, 18\}$,
$\{18, 5; 4, 18\}$.

The complement of $G(H)$, SRG$(81,60,45,42)$, has CR codes with the following intersection arrays:\\
eigenvalue $6$:
$\{36;18\}$,
$\{30;24\}$;\\
eigenvalue $-3$:
$\{56;7\}$,
$\{49;14\}$,
$\{42;21\}$,
$\{35;28\}$;\\
covering radius~$2$:
$\{60, 14; 1, 42\}$.

We have the following new parameters of CR-$1$ codes in
$H(n,3)$:
\begin{itemize}
    \item A $\{15;12\}$-CR code, $n=10$ (previous value $n=12$, lower bound $n\ge 10$).
\end{itemize}

\subsection{SRG(243,22,1,2)}\label{s:243}

There is only one Cayley SRG$(243,22,1,2)$ on $\FF_3^5$, the coset graph of the perfect Golay $[11,6,5]_3$ code, known as the Berlekamp--Van Lint--Seidel graph.
It has CR codes with the following intersection arrays:\\
eigenvalue $4$:
$\{16;2\}$, $\{14;4\}$, $\{12;6\}$, $\{10;8\}$;\\
eigenvalue $-5$:
$\underline{\{27-c;c\}}$, $c=\underline{5,{6},{7},8},9,\underline{10,11},{12},\underline{13}$;\\
covering radius~$2$:\\
$\{18, 5; 4, 18\}$,
$\{18, 6; 3, 18\}$,
$\{18, 7; 2, 18\}$,
$\{18, 8; 1, 18\}$,
$\{20, 18; 1, 6\}$,
$\{22, 20; 1, 2\}$
(the putative arrays
$\{20, 10; 2, 13\}$,
$\{21, 10; 2, 12\}$,
$\{22, 2; 4, 17\}$,
$\{22, 9; 2, 12\}$ 
were rejected with the ILP solver CPLEX in GAMS \cite{GAMS}).

CR-$1$ codes in SRG$(243,22,1,2)$
correspond to CR-$1$ codes
in the complement graph,
which is SRG$(243,220,199,200)$ (we did not search for CR-$2$ codes in this graph).
The corresponding intersection arrays are
the following:\\
 eigenvalue $4$: $\{216-c';c'\}$, $c'=40,48,...,104$;\\
 eigenvalue $-5$: $\{200;25\}$, $\{175;50\}$, $\{150;75\}$, $\{125;100\}$.

A $\{216-c';c'\}$-CR code with $c'=40$ in 
SRG$(243,220,199,200)$ (and hence, an independent $\{5;22\}$-CR set in SRG$(243,22,1,2)$) was previously 
known in the following context.
Such a code has size $45$ and inner valency~$44$, which means that it is a Delsarte clique. Such cliques in SRG$(243,220,199,200)$ were found 
in~\cite{BrouwerKoolenRiebeek:1998},
as a part of construction of another distance-regular graph.
Delsarte cliques in distance-regular graphs are known to be completely regular, see e.g.
\cite[Lemma 7.1]{Godsil93},
\cite[Proposition 1.44]{KroPot:Ch1}.

By covering, we find the following new parameters of CR-$1$ codes in
$H(n,3)$:
\begin{itemize}
    \item $\{27-c;c\}$-CR codes, $c=5,6,7,8,\,10,11,\,13$, $n=11$ (previous values were $n=13,12,13,12,\,13,12,\,13$, respectively; the lower bound is $n\ge 11$).
\end{itemize}

\subsection{SRG(81,30,9,12) and SRG(81,50,31,30)}\label{s:81}

There are exactly $2$ nonequivalent 
$[15,4,\{9,12\}]_3$ codes \cite{HamHel96}, with generator matrices as below. The corresponding SRGs, discussed in~\cite[\S10.29]{BvM:SRG},
have $116640$ and $5832$ automorphisms, respectively.
$$
H_1 = 
\left(
\begin{array}{@{\,}*{14}{c@{\ }}c@{\,}}
1&1&1&0&0&0&1&1&1&1&1&1&1&1&1\\
0&0&0&1&1&1&1&1&1&2&2&2&2&2&2\\
0&1&0&0&1&0&2&1&2&1&2&0&0&1&2\\
0&0&2&0&0&2&1&1&2&0&2&1&2&1&0 
\end{array} \right)\!,\ \ 
H_2 = 
\left(
\begin{array}{@{\,}*{14}{c@{\ }}c@{\,}}
1&1&1&0&0&0&1&1&1&1&1&1&1&1&1\\
0&0&0&1&1&1&1&1&1&2&2&2&2&2&2\\
0&1&0&0&1&0&2&1&2&0&1&2&0&1&2\\
0&0&2&0&0&2&1&1&2&0&2&1&2&1&0 
\end{array}  \right)\!.
$$
The parameters of CR-$1$ codes in these graphs are shown in Table~\ref{t:(81,30,9,12)}.
\begin{table}
\mbox{}\hfill
$
\begin{array}{c|c@{\ }c@{\ }c@{\ }c@{\ }c@{\ }c@{\ }c@{\ }c@{\ }c@{\ }c@{\ }c@{\ }c|c@{\ }c|}
\mathrm{SRG}(81,30,9,12)&\multicolumn{12}{|c|}{\mbox{eigenvalue $3$}}&
\multicolumn{2}{|c|}{\mbox{$-6$}}\\\hline
\raisebox{3ex}{$|\mathrm{Aut}|$}&
\rotatebox{90}{$\{25;2\}$} & 
\rotatebox{90}{$\{24;3\}$} & 
\rotatebox{90}{$\{23;4\}$} & 
\rotatebox{90}{$\{22;5\}$} &  
\rotatebox{90}{$\{21;6\}$} & 
\rotatebox{90}{$\{20;7\}$} & 
\rotatebox{90}{$\{19;8\}$} & 
\rotatebox{90}{$\{18;9\}$} & 
\rotatebox{90}{$\{17;10\}$} & 
\rotatebox{90}{$\{16;11\}$} & 
\rotatebox{90}{$\{15;12\}$} & 
\rotatebox{90}{$\{14;13\}$} & 
\rotatebox{90}{$\{24;12\}$} & 
\colorbox[gray]{0.95}{\rotatebox{90}{{\underline{$\{20;16\}$}}}} \\\hline
116640 & + & + & + & + & + & + & + & + & + & + & + & + & + & - \\
5832 & - & + & - & - & + & - & - & + & - & - & + & - & + & + \\
\hline
\raisebox{3ex}{$|\mathrm{Aut}|$}&
\rotatebox{90}{\colorbox[gray]{0.95}{\underline{$\{50;4\}$}}} & 
\rotatebox{90}{$\{48;6\}$} & 
\rotatebox{90}{\colorbox[gray]{0.95}{\underline{$\{46;8\}$}}} & 
\rotatebox{90}{$\{44;10\}$} &  
\rotatebox{90}{$\{42;12\}$} & 
\rotatebox{90}{$\{40;14\}$} & 
\rotatebox{90}{$\{38;16\}$} & 
\rotatebox{90}{$\{36;18\}$} & 
\rotatebox{90}{$\{34;20\}$} & 
\rotatebox{90}{$\{32;22\}$} & 
\rotatebox{90}{$\{30;24\}$} & 
\rotatebox{90}{$\{28;26\}$} & 
\rotatebox{90}{$\{30;15\}$} & 
\rotatebox{90}{{$\{25;20\}$}} \\\hline
\mathrm{SRG}(81,50,31,30)&\multicolumn{12}{|c|}{\mbox{eigenvalue $-4$}}&
\multicolumn{2}{|c|}{\mbox{$5$}}
\end{array}
$
\hfill\mbox{}
\caption{All parameters of CR-$1$ codes in the two nonisomorphic SRG$(81,30,9,12)$ Cayley graphs on $\FF_3^4$ and their SRG$(81,50,31,30)$ complements.
The symbol ``$+$'' (``$-$'') in the $i$th row and $j$th column means that a CR code with intersection array shown above ($2$nd row, $j$th column) exists (does not exist, respectively) in the SRG$(81,30,9,12)$ graph uniquely identified by the number of automorphisms ($i$th row, $1$st column); equivalently, a CR code with intersection array shown below ($5$th row, $j$th column) exists (does not exist, respectively) in the corresponding SRG$(81,50,31,30)$. The underlined
parameters contribute to the small-parameter table of CR codes in $H(n,3)$.}\label{t:(81,30,9,12)}
\end{table}
CR-$2$ codes: $G(H_1)$ has $\{30,6;3,24\}$, $\{30,20;1,12\}$;
$G(H_2)$ and both complement graphs
have only trivial (singleton, from one vertex) $\{30,20;1,12\}$-
and $\{50,18;1,30\}$-CR codes, respectively.

We have the following new parameters of CR-$1$ codes in
$H(n,3)$:
\begin{itemize}
    \item $\{20;16\}$-CR code, $n=15$ (previous value $n=16$, lower bound $n\ge 13$,
      a better value $n=14$ is obtained by doubling the parameters of the $\{10;8\}$-CR code found in Section~\ref{s:10,8});
    \item $\{46;8\}$-CR code, $n=25$ (previous value $n=26$, lower bound $n\ge 23$,
      a better value $n=24$ is obtained by doubling the parameters of the $\{23;4\}$-CR code found in Section~\ref{s:24});
    \item $\{50;4\}$-CR code, $n=25$ (previous value $n=26$, lower bound $n\ge 25$).
\end{itemize}

\subsection{SRG(81,24,9,6), SRG(81,56,37,42)}
\label{s:24}

There are exactly $2$ nonequivalent ``$\FF_3$-linear'' SRG$(81,24,9,6)$, 
VNO$^+_4(3)=G(H_1)$ with $93312$ automorphisms
and
OA$(3,\FF_9)$ with $23328$ automorphisms 
(see also \cite{Wallis71}), where


$$
H_1 = \left(\begin{array}{@{\,}c@{\ }c@{\ }c@{\ }c@{\ }c@{\ }c@{\ }c@{\ }c@{\ }c@{\ }c@{\ }c@{\ }c@{\,}}
      1&1&1&1&1&1&1&1&1&0&0&0\\
      0&0&0&1&1&1&1&2&0&1&0&0\\
      0&1&1&0&0&1&1&2&0&0&1&0\\
      1&0&1&0&1&0&1&2&0&0&0&1
\end{array}\right),\quad
H_2 = \left(\begin{array}{@{\,}c@{\ }c@{\ }c@{\ }c@{\ }c@{\ }c@{\ }c@{\ }c@{\ }c@{\ }c@{\ }c@{\ }c@{\,}}
      1&1&1&1&1&1&1&1&1&0&0&0\\
      0&0&0&1&1&1&1&2&0&1&0&0\\
      0&0&1&0&1&1&2&1&0&0&1&0\\
      1&2&0&1&0&1&1&0&0&0&0&1
\end{array}\right).$$

The parameters of CR-$1$ codes in these graphs are shown in Table~\ref{t:(81,24,9,6)}.
\begin{table}
\mbox{}\hfill
$ 
\begin{array}{c|c@{\ \ }c@{\ \ }c@{\ \ }c|c@{\,}c@{\,}c@{\ }c@{\ }c@{\ }c@{\ }c@{\ }c@{\ }c@{\ }c@{\ }c|}
\mathrm{SRG}(81,24,9,6)&\multicolumn{4}{|c|}{\mbox{eigenvalue $6$}}&
\multicolumn{11}{|c|}{\mbox{eigenvalue $-3$}}\\\hline
\raisebox{3ex}{$|\mathrm{Aut}|$}&
\rotatebox{90}{$\{16;2\}$} & 
\rotatebox{90}{$\{14;4\}$} & 
\rotatebox{90}{$\{12;6\}$} & 
\rotatebox{90}{$\{10;8\}$} &
\rotatebox{90}{$\{24;3\}$} &
\rotatebox{90}{\colorbox[gray]{0.95}{\underline{$\{23;4\}$}}} &
\rotatebox{90}{$\{22;5\}$} &
\rotatebox{90}{$\{21;6\}$} &
\rotatebox{90}{$\{20;7\}$} &
\rotatebox{90}{$\{19;8\}$} &
\rotatebox{90}{$\{18;9\}$} &
\rotatebox{90}{$\{17;10\}$} &
\rotatebox{90}{$\{16;11\}$} &
\rotatebox{90}{$\{15;12\}$} &
\rotatebox{90}{$\{14;13\}$} 
\\\hline
93312 & - & - & + & - 
& + & + & + & + & + & + & + & +  & + & + & +   \\
23328 & + & + & + & + 
& + & - & - & + & - & - & + & -  & - & + & -   \\
\hline
\raisebox{3ex}{$|\mathrm{Aut}|$}&
\rotatebox{90}{$\{56;7\}$} & 
\rotatebox{90}{$\{49;14\}$} & 
\rotatebox{90}{$\{42;21\}$} & 
\rotatebox{90}{$\{35;28\}$} &
\rotatebox{90}{$\{48;6\}$} &
\rotatebox{90}{$\{46;8\}$} &
\rotatebox{90}{$\{44;10\}$} &
\rotatebox{90}{$\{42;12\}$} &
\rotatebox{90}{$\{40;14\}$} &
\rotatebox{90}{$\{38;16\}$} &
\rotatebox{90}{$\{36;18\}$} &
\rotatebox{90}{$\{34;20\}$} &
\rotatebox{90}{$\{32;22\}$} &
\rotatebox{90}{$\{30;24\}$} &
\rotatebox{90}{$\{28;26\}$}
 \\\hline
\mathrm{SRG}(81,56,37,42)&\multicolumn{4}{|c|}{\mbox{eigenvalue $-7$}}&
\multicolumn{11}{|c|}{\mbox{eigenvalue $2$}}
\end{array}
$
\hfill\mbox{}
\caption{All parameters of CR-$1$ codes in the two nonisomorphic SRG$(81,24,9,6)$ Cayley graphs on $\FF_3^4$ and their SRG$(81,56,37,42)$ complements. The notations are similar to Table~\ref{t:(81,30,9,12)}.}\label{t:(81,24,9,6)}
\end{table}


CR-$2$ codes: $\{18,6;3,18\}$, $\{24,14;1,6\}$ in each of $G(H_1)$ and
 $G(H_2)$;
$\{56,18;1,42\}$ in $\overline{G(H_1)}$ and  $\overline{G(H_2)}$.

We have the following new parameters of CR-$1$ codes in
$H(n,3)$:
\begin{itemize}
    \item $\{23;4\}$-CR code, $n=12$ (previous value $n=13$, lower bound $n\ge 12$).
\end{itemize}

\subsection{SRG(81,32,13,12), SRG(81,48,27,30)}
\label{s:48}

There are exactly four $[16,4,\{9,12\}]_3$ codes.
The corresponding graphs
have $186624$ automorphisms
(VO$^+_4(3)=G(H_1)$, see e.g.~\cite[\S3.3]{BvM:SRG}), $5184$ 
(OA$(4,\FF_9)=G(H_4)$),
$11664$ and $11664$ ($G(H_2)$ and $G(H_3)$) automorphisms, respectively.
$$ 
 H_1 = \left( \begin{array}{@{\,}*{15}{c@{\ }}c@{\,}}
 0&1&1&1&1&1&1&1&1&1&1&1&1&0&0&0 \\
 1&0&0&0&0&1&1&1&1&1&2&2&0&1&0&0 \\
 1&1&1&1&2&0&0&0&1&2&0&1&0&0&1&0 \\
 2&0&1&2&0&0&1&2&1&1&0&1&0&0&0&1 \\
 \end{array} \right)\!,
 \ \ 
 H_2 = \left( \begin{array}{@{\,}*{15}{c@{\ }}c@{\,}}
 0&1&1&1&1&1&1&1&1&1&1&1&1&0&0&0 \\
 1&0&0&0&0&1&1&1&1&1&2&2&0&1&0&0 \\
 1&1&1&1&2&0&0&0&1&2&0&1&0&0&1&0 \\
 1&0&1&2&0&0&1&2&1&0&1&1&0&0&0&1 \\
 \end{array} \right)\!,
 $$ $$
 H_3 = \left( \begin{array}{@{\,}*{15}{c@{\ }}c@{\,}}
 0&1&1&1&1&1&1&1&1&1&1&1&1&0&0&0 \\
 1&0&0&0&0&1&1&1&1&1&2&2&0&1&0&0 \\
 1&1&1&1&2&0&0&0&1&2&0&1&0&0&1&0 \\
 0&0&1&2&1&0&1&2&0&1&2&2&0&0&0&1 \\
 \end{array} \right)\!,
 \ \ 
 H_4 = \left( \begin{array}{@{\,}*{15}{c@{\ }}c@{\,}}
 0&1&1&1&1&1&1&1&1&1&1&1&1&0&0&0 \\
 1&0&0&0&0&1&1&1&1&1&2&2&0&1&0&0 \\
 1&0&0&1&2&0&0&1&1&2&1&1&0&0&1&0 \\
 1&1&2&1&0&0&1&0&1&0&1&2&0&0&0&1 \\
 \end{array} \right)\!. 
 $$

Each of $G(H_i)$, $i=1,2,3,4$, has CR with the following intersection arrays:\\
eigenvalue $5$:
$\{24;3\}$,
$\{21;6\}$,
$\{18;9\}$,
$\{15;12\}$;\\
eigenvalue $-4$:
$\{32;4\}$,
$\{28;8\}$,
$\{24;12\}$,
$\{20;16\}$;\\
covering radius~$2$: $\{32, 18; 1, 12\}$ (the singleton code).

The complements, SRG$(81, 48, 27, 30)$, have CR with the following  intersection arrays:\\
eigenvalue $3$:
$\{40;5\}$,
$\{35;10\}$,
$\{30;15\}$,
$\{25;20\}$;\\
eigenvalue $-6$:
$\{48;6\}$,
$\{42;12\}$,
$\{36;18\}$,
$\{30;24\}$;\\
covering radius~$2$:
$\{48, 20; 1, 30\}$ (the singleton code).

\begin{itemize}
    \item By searching in these graphs, we have not found new parameters of CR codes in $H(n,3)$. 
    We were looking for $\{46;8\}$.
\end{itemize}

\subsection{SRG(81,40,19,20)}
There are exactly four $[20,4,\{12,15\}]_3$ codes, with generator matrices 
$$ H_1 = \left( \begin{array}{@{\,}*{19}{c@{\,}}c@{\,}}
0&0&1&1&1&1&1&1&1&1&1&1&1&1&1&1&1&0&0&0\\
1&1&0&0&0&0&0&1&1&1&1&1&1&2&2&2&0&1&0&0\\
1&1&0&1&1&1&2&0&0&0&1&1&2&0&1&2&0&0&1&0\\
1&2&1&0&1&2&0&0&1&2&0&1&0&1&1&2&0&0&0&1\\
 \end{array} \right)\!,\ \ 
 H_2 = \left( \begin{array}{@{\,}*{19}{c@{\,}}c@{\,}}
0&0&1&1&1&1&1&1&1&1&1&1&1&1&1&1&1&0&0&0\\
1&1&0&0&0&0&0&1&1&1&1&1&1&2&2&2&0&1&0&0\\
1&1&0&0&1&1&2&0&0&1&1&2&2&0&1&1&0&0&1&0\\
0&1&1&2&0&2&1&0&1&0&2&0&1&2&1&2&0&0&0&1\\
 \end{array} \right)\!, $$
$$ H_3 = \left( \begin{array}{@{\,}*{19}{c@{\,}}c@{\,}}
0&0&1&1&1&1&1&1&1&1&1&1&1&1&1&1&1&0&0&0\\
1&1&0&0&0&0&0&1&1&1&1&1&1&2&2&2&0&1&0&0\\
1&1&0&0&1&1&2&0&0&1&1&2&2&0&1&1&0&0&1&0\\
1&2&1&2&0&1&2&0&1&0&1&0&1&2&0&1&0&0&0&1\\
 \end{array} \right)\!, \ \ 
 H_4 = \left( \begin{array}{@{\,}*{19}{c@{\,}}c@{\,}}
0&0&1&1&1&1&1&1&1&1&1&1&1&1&1&1&1&0&0&0\\
1&1&0&0&0&0&0&1&1&1&1&1&1&2&2&2&0&1&0&0\\
1&1&0&0&1&1&2&0&0&1&1&2&2&0&1&1&0&0&1&0\\
1&2&1&2&0&1&0&1&2&0&1&1&2&0&0&1&0&0&0&1\\
 \end{array} \right)\!; $$
the corresponding graphs (with 38880, 12960, 10368, 1944 automorphisms, respectively)
are self-complementary and have CR
with the following intersection arrays (all these parameters are known in $H(20,3)$):
\\
eigenvalue $4$:
$\{32;4\}$, $\{28;8\}$, $\{24;12\}$, $\{20;16\}$;\\
eigenvalue $-5$:
$\{40;5\}$, $\{35;10\}$, $\{30;15\}$, $\{25;20\}$;\\
covering radius~$2$:
$\{40, 20; 1, 20\}$ (the singleton code).\\

\subsection{SRG(64,27,10,12), SRG(64,36,20,20)}
\label{s:27}

The graph denoted as VO$^-_6(2)$ in \cite{BvM:SRG} has parameters
SRG$(64,27,10,12)$ and $3317760$ automorphisms.
In total, there are $5$ nonequivalent   SRG$(64,27,10,12)$ that are Cayley on $\FF_2^6$, with $3317760$, $73728$, $10240$, $7680$, and $1536$  automorphisms, respectively. The corresponding $6\times 27$ matrices are given below (in some cases, we skip the first rows that are shared with the previous matrix).
$$ H_1=\left( \begin{array}{@{\,}*{26}{c@{\ }}c@{\,}}
  0&0&0&0&0&0&1&1&1&1&1&1&1&1&1&1&1&1&1&1&1&1&0&0&0&0&0 \\
  0&1&1&1&1&1&0&0&0&0&0&0&0&0&0&1&1&1&1&1&1&0&1&0&0&0&0 \\
  1&0&0&0&1&1&0&0&0&0&0&1&1&1&1&0&0&1&1&1&1&0&0&1&0&0&0 \\
  1&0&0&1&0&1&0&0&1&1&1&0&0&0&1&0&1&0&1&1&1&0&0&0&1&0&0 \\
  1&0&1&0&0&1&1&1&0&0&1&0&0&1&0&0&1&1&0&1&1&0&0&0&0&1&0 \\
  1&1&0&0&0&1&0&1&0&1&1&0&1&1&1&1&0&0&0&0&1&0&0&0&0&0&1
\end{array} \right) $$
$$ H_2=\left( \begin{array}{@{\,}*{26}{c@{\ }}c@{\,}}
 \multicolumn{10}{l}{\ldots} \\
  1&0&0&1&0&1&0&0&1&1&1&0&0&0&1&0&1&0&1&1&1&0&0&0&1&0&0 \\
  0&0&1&0&1&1&0&1&0&1&1&0&0&1&1&1&0&0&0&1&1&0&0&0&0&1&0 \\
  0&1&0&0&1&1&1&0&1&0&1&0&1&0&1&1&0&0&1&0&1&0&0&0&0&0&1
\end{array} \right) $$
$$ H_3=\left( \begin{array}{@{\,}*{26}{c@{\ }}c@{\,}}
 \multicolumn{10}{l}{\ldots} \\
  0&0&0&1&1&1&0&0&1&1&1&0&0&1&1&0&1&0&0&1&1&0&0&0&1&0&0 \\
  0&0&1&1&0&1&0&1&0&0&1&0&1&1&1&1&0&0&1&0&1&0&0&0&0&1&0 \\
  1&1&0&1&0&0&1&0&0&1&0&0&1&0&1&1&0&1&1&0&1&0&0&0&0&0&1
\end{array} \right) $$
$$ H_4=\left( \begin{array}{@{\,}*{26}{c@{\ }}c@{\,}}
 \multicolumn{10}{l}{\ldots} \\
  1&1&0&0&1&0&1&0&0&1&1&0&0&0&1&1&0&1&0&1&1&0&0&0&0&0&1
\end{array} \right) $$
$$ H_5=\left( \begin{array}{@{\,}*{26}{c@{\ }}c@{\,}}
      \multicolumn{10}{l}{\ldots} \\
  1&1&0&0&1&0&1&0&0&1&0&0&1&0&1&1&1&0&0&1&1&0&0&0&0&0&1 
\end{array} \right) $$

The parameters of CR-$1$ codes in these graphs are shown in Table~\ref{t:(64,27,10,12)}.
\begin{table}
\mbox{}\hfill
$
\begin{array}{@{\ }c|c@{\,}c@{\,}c@{\,}c|c@{\,}c@{\,}c|}
\mathrm{SRG}(64,27,10,12)&\multicolumn{4}{|c|}{\mbox{eigenvalue $3$}}&
\multicolumn{3}{|c|}{\mbox{eigenvalue~$-5$}}\\\hline
\rotatebox{0}{$|\mathrm{Aut}|$}&
\rotatebox{0}{$\{21;3\}$} & 
\rotatebox{0}{$\{18;6\}$} & 
\rotatebox{0}{$\{15;9\}$} & 
\rotatebox{0}{$\{12;12\}$} &  
\rotatebox{0}{$\{24;8\}$} & 
\rotatebox{0}{$\{20;12\}$} & 
\rotatebox{0}{$\{16;16\}$} \\\hline
3317760 & - & + & + & + & + & + & + \\\hline
\parbox[][][l]{7em}{73728, 10240,\\ 7680, 1536} & + & + & + & + & + & + & + \\
\hline
\rotatebox{0}{$|\mathrm{Aut}|$}&
\rotatebox{0}{$\{35;5\}$} & 
\rotatebox{0}{$\{30;10\}$} & 
\rotatebox{0}{$\{25;15\}$} & 
\rotatebox{0}{$\{20;20\}$} &  
\rotatebox{0}{$\{24;8\}$} & 
\rotatebox{0}{$\{20;12\}$} & 
\rotatebox{0}{$\{16;16\}$}  
\\\hline
\mathrm{SRG}(64, 36, 20, 20)&\multicolumn{4}{|c|}{\mbox{eigenvalue $-4$}}&
\multicolumn{3}{|c|}{\mbox{eigenvalue $4$}}
\end{array}
$
\hfill\mbox{}
\caption{All parameters of CR-$1$ codes in the five nonisomorphic SRG$(64,27,10,12)$ Cayley graphs on $\FF_2^6$ and their SRG$(64, 36, 20, 20)$ complements. The notations are similar to Table~\ref{t:(81,30,9,12)}.}\label{t:(64,27,10,12)}
\end{table}

Parameters of CR-$2$ codes in each
$G(H_i)$, $i=1,2,3,4,5$:
$\{24,4;4,24\}$,
$\{24,6;2,24\}$,
$\{27,16;1,12\}$;
in each $\overline{G(H_i)}$:
$\{36, 15; 1, 20\}$.

\begin{itemize}
        \item By searching in these graphs, we have not found new parameters of CR codes in $H(n,2)$. 
    We were looking for  $\{27;5\}$, $\{21;11\}$, $\{19;13\}$, $\{17;15\}$.
\end{itemize}

\subsection{SRG(64,28,12,12) and SRG(64,35,18,20)}
\label{s:28}

There are six SRG$(64,28,12,12)$ that are Cayley on $\FF_2^6$, with $2580480$, $24576$, $7680$, $6144$,  $5376$, and $1536$  automorphisms, respectively.

$$ H_1=\left( \begin{array}{@{\,}*{27}{c@{\ }}c@{\,}}
  0&0&0&0&0&0&0&1&1&1&1&1&1&1&1&1&1&1&1&1&1&1&1&0&0&0&0&0 \\
  0&0&1&1&1&1&1&0&0&0&0&0&0&0&0&0&1&1&1&1&1&1&0&1&0&0&0&0 \\
  0&1&0&0&0&1&1&0&0&0&0&0&1&1&1&1&0&0&1&1&1&1&0&0&1&0&0&0 \\
  1&0&0&1&1&0&0&0&0&0&1&1&0&0&1&1&1&1&0&0&1&1&0&0&0&1&0&0 \\
  0&1&1&0&0&0&1&0&1&1&0&1&0&0&0&1&0&1&1&1&0&1&0&0&0&0&1&0 \\
  1&0&1&0&1&1&1&1&0&1&0&0&0&1&0&1&0&0&0&1&0&1&0&0&0&0&0&1
\end{array} \right) $$
$$ H_2=\left( \begin{array}{@{\,}*{27}{c@{\ }}c@{\,}}
  0&0&0&0&0&0&0&1&1&1&1&1&1&1&1&1&1&1&1&1&1&1&1&0&0&0&0&0 \\
  1&1&1&1&1&1&1&0&0&0&0&0&0&0&1&1&1&1&1&1&1&1&0&1&0&0&0&0 \\
  0&0&0&0&1&1&1&0&0&0&0&1&1&1&0&0&0&1&1&1&1&1&0&0&1&0&0&0 \\
  0&0&0&1&0&1&1&0&0&1&1&0&0&1&0&1&1&0&0&1&1&1&0&0&0&1&0&0 \\
  0&1&1&0&0&0&1&0&1&0&0&0&1&1&1&0&1&1&1&0&0&1&0&0&0&0&1&0 \\
  1&0&1&0&0&0&1&1&0&0&1&0&0&1&1&1&1&0&1&0&1&0&0&0&0&0&0&1
\end{array} \right) $$
$$ H_3=\left( \begin{array}{@{\,}*{27}{c@{\ }}c@{\,}}
 \multicolumn{10}{l}{\ldots} \\
  0&0&0&1&0&1&1&0&0&1&1&0&0&1&0&1&1&0&0&1&1&1&0&0&0&1&0&0 \\
  0&1&1&0&0&0&1&0&1&0&0&0&1&1&1&0&1&1&1&0&0&1&0&0&0&0&1&0 \\
  1&0&1&0&0&0&1&1&1&0&1&0&0&0&0&1&1&0&1&0&1&1&0&0&0&0&0&1
\end{array} \right) $$
$$ H_4=\left( \begin{array}{@{\,}*{27}{c@{\ }}c@{\,}}
 \multicolumn{10}{l}{\ldots} \\
  0&0&1&1&0&0&1&0&0&1&1&0&0&1&1&1&1&0&0&0&1&1&0&0&0&1&0&0 \\
  0&1&0&0&0&1&1&1&1&0&0&0&1&0&0&1&1&0&1&1&0&1&0&0&0&0&1&0 \\
  1&0&0&1&0&0&1&0&1&0&1&1&1&1&0&0&1&0&0&1&0&1&0&0&0&0&0&1
\end{array} \right) $$
$$ H_5=\left( \begin{array}{@{\,}*{27}{c@{\ }}c@{\,}}
 \multicolumn{10}{l}{\ldots} \\
  0&1&0&0&0&1&1&1&1&0&0&0&1&0&0&1&1&0&1&1&0&1&0&0&0&0&1&0 \\
  1&1&0&1&0&0&0&0&1&0&1&0&0&1&1&0&1&1&0&1&1&0&0&0&0&0&0&1
\end{array} \right) $$
$$ H_6=\left( \begin{array}{@{\,}*{27}{c@{\ }}c@{\,}}
 \multicolumn{10}{l}{\ldots} \\
  0&1&0&0&0&1&1&0&1&0&1&0&1&0&0&1&1&0&0&1&1&1&0&0&0&0&1&0 \\
  1&0&0&1&1&0&0&1&0&0&0&1&1&0&1&0&1&0&1&1&0&1&0&0&0&0&0&1
\end{array} \right) $$
The parameters of CR-$1$ codes in these graphs are shown in Table~\ref{t:(64,28,12,12)}.
\begin{table}[ht]
\mbox{}\hfill
$
\begin{array}{c|cccc|ccccccc|}
\mathrm{SRG}(64,28,12,12)&\multicolumn{4}{|c|}{\mbox{eigenvalue $4$}}&
\multicolumn{7}{|c|}{\mbox{eigenvalue $-4$}}\\\hline
\raisebox{3ex}{$|\mathrm{Aut}|$}&\rotatebox{90}{$\{21;3\}$} & 
\rotatebox{90}{$\{18;6\}$} & 
\rotatebox{90}{$\{15;9\}$} & 
\rotatebox{90}{$\{12;12\}$} &  
\rotatebox{90}{$\{28;4\}$} & 
\rotatebox{90}{$\{26;6\}$} & 
\rotatebox{90}{$\{24;8\}$} & 
\rotatebox{90}{$\{22;10\}$} & 
\rotatebox{90}{$\{20;12\}$} & 
\rotatebox{90}{$\{18;14\}$} & 
\rotatebox{90}{$\{16;16\}$} \\\hline
2580480,6144,5376,1536  & + & + & + & +  & + & - & + & - & + & - & + \\
24576 & - & + & - & +  & + & + & + & + & + & + & + \\
7680 & - & + & + & +  & + & - & + & - & + & - & +  \\
\hline
\raisebox{3ex}{$|\mathrm{Aut}|$}&\rotatebox{90}{$\{35;5\}$} & 
\rotatebox{90}{$\{30;10\}$} & 
\rotatebox{90}{$\{25;15\}$} & 
\rotatebox{90}{$\{20;20\}$} &  
\rotatebox{90}{$\{28;4\}$} & 
\rotatebox{90}{$\{26;6\}$} & 
\rotatebox{90}{$\{24;8\}$} & 
\rotatebox{90}{$\{22;10\}$} & 
\rotatebox{90}{$\{20;12\}$} & 
\rotatebox{90}{$\{18;14\}$} & 
\rotatebox{90}{$\{16;16\}$} \\\hline
\mathrm{SRG}(64,35,18,20)&\multicolumn{4}{|c|}{\mbox{eigenvalue $-5$}}&
\multicolumn{7}{|c|}{\mbox{eigenvalue $3$}}
\end{array}
$
\hfill\mbox{}
\caption{All parameters of CR-$1$ codes in the six nonisomorphic SRG$(64,28,12,12)$ Cayley graphs on $\FF_2^6$ and their SRG$(64,35,18,20)$ complements. The notations are similar to Table~\ref{t:(81,30,9,12)}.}\label{t:(64,28,12,12)}
\end{table}

Parameters of CR-$2$ codes in each
$G(H_i)$, $i=1,...,6$:
$\{24,4;4,24\}$,
$\{28,15;1,12\}$;
in each $\overline{G(H_i)}$:
$\{35, 16; 1, 20\}$.

\begin{itemize}
        \item By searching in these graphs, we have not found new parameters of CR codes in $H(n,2)$. 
    We were looking for $\{17;15\}$.
\end{itemize}

\section
{Summary and conclusions} \label{s:sum}
Table~\ref{tab:ias} summarizes all new basic parameters of completely regular 
codes in Hamming graphs, 
where `basic' means that from them we can obtain an infinite number of new parameters with the constructions in Propositions~\ref{p:*s}, \ref{p:split}, and Theorem~\ref{th:1to2}.

\begin{table}[ht]
    \centering
\def\ffrac#1#2{{#1}\cdot{#2}}
\begin{tabular}{@{}c@{\ }|c@{\ }|c@{\ }|c@{\ }|c@{\ }|c|l@{\ }}
    I.A. & $q$ &  density & $n$($n_{\mathrm{old}}$) & l.bound & {$\tau = (b+c)/qn$} & \ in graph \\ \hline \\[-1.5ex]
 $\{66,15;1,54\}$ & $2$ &  $\ffrac{3}{2^{-8}}$ & $66(-)$ & $\ge 66$ & --- & $\frac12H(12,2)$, Section~\ref{s:Hadamard}\\[1mm]
 $\{10;8\}$ & $3$ &  $\ffrac{4}{3^{-2}}$ & $7(8)$ & $\ge 7$ & $6/7 \sim 0.857$ & $\Cay(\FF_3^4)$, Section~\ref{s:10,8}\\[1mm]
  $\{23;4\}$ & $3$ & $\ffrac{4}{3^{-3}}$ &  $12(13)$ & $\ge 12$ & $3/4 = 0.75$ &${\mathrm{VNO}^+_4(3)}$\\[1mm]
  $\{22;5\}$ & $3$ & $\ffrac{5}{3^{-3}}$ &  $11(13)$ & $\ge 11$ & $9/11 \sim 0.818$ & \parbox{4cm}{Berlekamp--Van Lint--Seidel graph}\\[1mm]
  $\{21;6\}$ & $3$ & $\ffrac{2}{3^{-2}}$ &  $11(12)$ & $\ge 11$ & $9/11 \sim 0.818$ & B.--v.L.--S. \cite[10.55]{BvM:SRG}  \\[1mm]
  $\{20;7\}$ & $3$ & $\ffrac{7}{3^{-3}}$ &  $11(13)$ & $\ge 11$ & $9/11 \sim 0.818$ & B.--v.L.--S. \cite[10.55]{BvM:SRG}   \\[1mm]
  $\{19;8\}$ & $3$ & $\ffrac{8}{3^{-3}}$ &  $11(12)$ & $\ge 11$ & $9/11 \sim 0.818$ & B.--v.L.--S. \cite[10.55]{BvM:SRG}   \\[1mm]
 $\{17;10\}$ & $3$ & $\ffrac{10}{3^{-3}}$ &  $11(13)$ & $\ge 11$ & $9/11 \sim 0.818$ & B.--v.L.--S. \cite[10.55]{BvM:SRG}   \\[1mm]
 $\{16;11\}$ & $3$ & $\ffrac{11}{3^{-3}}$ &  $11(12)$ & $\ge 11$ & $9/11 \sim 0.818$ & B.--v.L.--S. \cite[10.55]{BvM:SRG}   \\[1mm]
 $\{15;12\}$ & $3$ & $\ffrac{12}{3^{-3}}$ &  $10(12)$ & $\ge 10$ & $9/10=0.9$ & \parbox[r]{4cm}{Brouwer--Haemers graph \cite[10.28]{BvM:SRG}}\\[1mm]
 $\{14;13\}$ & $3$ & $\ffrac{13}{3^{-3}}$ &  $11(13)$ & $\ge 11$ & $9/11 \sim 0.818$ & B.--v.L.--S. \cite[10.55]{BvM:SRG}   \\[1mm]
 $\{20;16\}$ & $3$ & $\ffrac{4}{3^{-2}}$ &  $15(16)$ & $\ge 13$ & $4/5 =0.8$ & $\Delta\ne{\mathrm{VNO}^-_4(3)}$ \cite[10.29]{BvM:SRG}\\[1mm]
            &   & & $14(16)$ & $\ge 13$ & $6/7 \sim 0.857$ & $2\times\{10;8\}$ \\[1mm]
  $\{28;8\}$ & $3$ & $\ffrac{2}{3^{-2}}$ &  $15(16)$ & $\ge 14$ & $4/5 =0.8$ & $\Cay(\FF_3^4)$, Section~\ref{s:28-8}\\[1mm]
 $\{35;10\}$ & $3$ & $\ffrac{2}{3^{-2}}$ &  $19(20)$ & $\ge 18$ & $15/19 \sim 0.789$ & $\Cay(\FF_3^4)$, Section~\ref{s:35-10}\\[1mm]
 $\{25;20\}$ & $3$ & $\ffrac{4}{3^{-2}}$ &  $19(20)$ & $\ge 16$ & $15/19 \sim 0.789$ & $\Cay(\FF_3^4)$, Section~\ref{s:35-10}\\[1mm]
  $\{46;8\}$ & $3$ & $\ffrac{4}{3^{-3}}$ &  $25(26)$ & $\ge 23$ & $18/25=0.72$ & $\overline{\mathrm{VNO}^-_4(3)}$ \cite[10.29]{BvM:SRG}\\[1mm]
   &  &  & $24(26)$ & $\ge 23$ & $3/4=0.75$ & $2\times\{23;4\}$  \\[1mm]
  $\{50;4\}$ & $3$ & $\ffrac{2}{3^{-3}}$ &  $25(26)$ & $\ge 25$ & $18/25=0.72$ & $\overline{\mathrm{VNO}^-_4(3)}$  \cite[10.29]{BvM:SRG}\\[1mm]
 $\{21,4;2,21\}$ & $3$ & $\ffrac{2}{3^{-3}}$ &  $11(-)$ & $\ge 11$ & --- & $\Cay(\FF_3^4)$, Section~\ref{s:21,4,2,21}  
\end{tabular}
    \caption{New basic parameters of CR codes in $H(n,q)$, $q=2,3$.
    The `density' of a code (the third column) is the cardinality divided by $q^n$.}
    \label{tab:ias}
\end{table}
The value of $\tau$ for new codes shows that Conjecture~1 is disproved 
even in its weak form~(ii), because for most new codes we have $\tau>q/(q+1)=3/4$.
The highest value of $\tau$ we know now for unbalanced CR-$1$ in the $3$-ary case is $9/10$.
However, the highest known value of $\tau$ for CR-$1$ of cardinality less than
$q^{n-1}$ (we will say that such a set is \emph{strictly unbalanced}),
$q=3$, is only $9/11$. So, we separate the problem into two questions.

Problem 1. For a given $q$, is the relative strength of a simple unbalanced orthogonal array (in particular, CR-$1$) bounded by some constant smaller than~$1$? For $q=2$, the known solution is~$2/3$~\cite{FDF:CorrImmBound}.

Problem 2. For a given $q$, is the relative strength of a strictly unbalanced orthogonal array (in particular, CR-$1$, but in general we do not require simplicity) bounded by some constant smaller than~$1$? For $q=2$, the known solution is again $2/3$ \cite{Khalyavin:2010.en} (a generalization of \cite{FDF:CorrImmBound} to non-simple orthogonal arrays).

Another conclusion is that there are nonbinary ($q>2$) orthogonal arrays 
attaining bound~\eqref{eq:BF} whose cardinality is not a power of~$q$ (nor a power of the prime divisor of~$q$, which means that such orthogonal arrays cannot be constructed as linear or additive codes).
Examples are OA$(5\cdot 3^8,11,3,8)$ (from $\{22;5\}$-CR in $H(11,3)$),
 OA$(2\cdot 3^{22},25,3,17)$ (from $\{50;4\}$-CR in $H(25,3)$),
 and an infinite number of orthogonal arrays (independent CR-$1$ sets) 
 constructed from these two examples by Propositions~\ref{p:*s}(ii)
 and~\ref{p:split}.

The only known binary distance-$3$ completely regular code whose cardinality is not a power of~$2$ (Section~\ref{s:Hadamard}) and 
the first known ternary CR-$2$ codes (Section~\ref{s:21,4,2,21} and Theorem~\ref{th:1to2}) whose cardinality is not a power of~$3$ are also worth mentioning. Finding more examples (including larger~$q$), in particular, generalizing the construction in Section~\ref{s:the}, would be very interesting.

\section*{Data availability statement}
The data that support the findings of this study (all found CR codes in the considered Cayley graphs) are openly available in Zenodo at \url{https://doi.org/10.5281/zenodo.21822086}
and partially (only the codes that contribute to the small-parameter tables of CR codes in $H(n,3)$) in IEEE DataPort~\cite{Perfect-related}, among other data related to completely regular codes.


\providecommand\href[2]{#2} \providecommand\url[1]{\href{#1}{#1}} \def\DOI#1{{\href{https://doi.org/#1}{https://doi.org/#1}}}\def\DOIURL#1#2{{\href{https://doi.org/#2}{https://doi.org/#1}}}

\end{document}